\documentclass[onefignum,onetabnum]{siamart190516}
\usepackage{layout}
\usepackage[sort,nocompress]{cite}

\usepackage{graphicx}
\usepackage{amssymb}
\usepackage{amsfonts}
\usepackage{url}
\usepackage{xfrac}
\usepackage{xcolor}
\usepackage{todonotes}
\usepackage{mathtools}
\usepackage{siunitx}
\usepackage{mathrsfs}  

\newsiamremark{remark}{Remark}
\newsiamremark{example}{Example}
\newsiamremark{hypothesis}{Hypothesis}
\crefname{hypothesis}{Hypothesis}{Hypotheses}
\newsiamthm{claim}{Claim}

\usepackage{bm}
\providecommand{\mathbold}[1]{\bm{#1}}
\newcommand{\R}{\mathbb{R}}
\newcommand{\N}{\mathbb{N}}
\newcommand{\Q}{\mathbb{Q}}
\newcommand{\Z}{\mathbb{Z}}
\newcommand{\cA}{\mathcal{A}}

\newcommand{\emat}{\cA}

\newcommand{\evec}{\vct{\alpha}}
\newcommand{\bevec}{\vct{\beta}}
\newcommand{\e}{\mathrm{e}}

\newcommand{\eentry}{\alpha}

\newcommand{\vct}[1]{\bm{#1}}
\newcommand{\mtx}[1]{\mathbold{#1}}

\newcommand{\foralll}{\text{ for all }}

\DeclareMathOperator{\cl}{cl}

\DeclareMathOperator{\supp}{supp}
\DeclareMathOperator{\invsupp}{invsupp}
\DeclareMathOperator{\spann}{span}

\newcommand{\Csage}{\mathcal{C}}
\newcommand{\csage}{\mathcal{C}}

\newcommand{\positiv}{Positivstellensatz}
\newcommand{\positive}{Positivstellens\"atze}
\newcommand{\sigation}{signomialization}
\newcommand{\polation}{polynomialization}
\newcommand{\Sigation}{Signomialization}

\newcommand{\diff}{\,\mathrm{d}}
\DeclareMathOperator{\Cont}{Cont}

\DeclareMathOperator{\Exp}{exp}

\newcommand{\set}[1]{\mathsf{#1}}
\newcommand{\X}{\set{X}}
\newcommand{\genExps}{A}

\theoremstyle{plain}
\theoremheaderfont{\normalfont\sc}
\theorembodyfont{\normalfont\itshape}
\theoremseparator{.}
\theoremsymbol{}
\newtheorem*{positivIntro}{Conditional SAGE \positive}

\theoremstyle{plain}
\theoremheaderfont{\normalfont\sc}
\theorembodyfont{\normalfont\itshape}
\theoremseparator{.}
\theoremsymbol{}
\newtheorem*{uppersIntro}{Upper Bounds}

\definecolor{DarkGreen}{rgb}{0,0.65,0}

\definecolor{NiceBlue}{rgb}{0.2,0.2,0.75}
\newcommand{\struc}[1]{{#1}}  

\title{Algebraic perspectives on signomial optimization\thanks{Authors listed alphabetically. \funding{R. Murray was supported by an NSF Graduate Research Fellowship}}}

\author{
    Mareike Dressler\thanks{
            Department of Mathematics, University of California, San Diego
          (\email{mdressler@ucsd.edu}).
    }
\and Riley Murray\thanks{
        Department of Electrical Engineering and Computer Sciences, University of California, Berkeley (\email{rjmurray@berkeley.edu}).
        \textit{Previously}: Department of Computing and Mathematical Sciences, California Institute of Technology.
  }
}

\begin{document}

\maketitle

\begin{abstract}
    Signomials are obtained by generalizing polynomials to allow for arbitrary real exponents.
    This generalization offers great expressive power, but has historically sacrificed the organizing principle of ``degree'' that is central to polynomial optimization theory.
    We reclaim that principle here through the concept of signomial rings, which we use to derive complete convex relaxation hierarchies of upper and lower bounds for signomial optimization via sums of arithmetic-geometric exponentials (SAGE) nonnegativity certificates.
    The Positivstellensatz underlying the lower bounds relies on the concept of \textit{conditional SAGE} and removes regularity conditions required by earlier works, such as convexity and Archimedeanity of the feasible set.
    Through worked examples we illustrate the practicality of this hierarchy in areas such as chemical reaction network theory and chemical engineering.
    These examples include comparisons to direct global solvers (e.g., BARON and ANTIGONE) and the Lasserre hierarchy (where appropriate).
    The completeness of our hierarchy of upper bounds follows from a generic construction whereby a \positiv{} for signomial nonnegativity over a compact set provides for arbitrarily strong \textit{outer approximations} of the corresponding cone of nonnegative signomials.
    While working toward that result, we prove basic facts on the existence and uniqueness of solutions to signomial moment problems.
\end{abstract}



\section{Introduction}

\positive{} in real algebraic geometry and moment problems in functional analysis are cornerstones of contemporary polynomial optimization theory.
We prove new results in these areas for a class of functions known as \textit{signomials}.
In much of the literature, signomials are parameterized as generalized polynomials $\vct{t} \mapsto \sum_{\evec\in \genExps}c_{\evec}t_1^{\eentry_1}\cdots t_{n}^{\eentry_n}$, which are well-defined for any $\evec\in\R^n$ provided $\vct{t} >\vct{0}$.
We study signomials under an equivalent parameterization obtained by a logarithmic change of variables.
That is, we consider a signomial supported on a set $\genExps \subset \R^n$ as a finite real-linear combination $\vct{x} \mapsto \sum_{\evec\in \genExps}c_{\evec}\exp\langle\evec,\vct{x}\rangle$.

Our main contributions in this article are constructions of arbitrarily strong inner and outer approximations of cones of signomials that are nonnegative on a prescribed compact set.
These constructions serve as the basis of our hierarchies of convex programs for nonconvex signomial optimization.
In the context of signomial nonnegativity problems, one can move between the exponential and generalized-polynomial parameterizations by identifying a signomial with its coefficient vector. 

\subsection{Why study signomials?\nopunct}

The most prominent use of signomials in applied mathematics is \textit{signomial programming} -- the minimization of a signomial subject to finitely many signomial inequality constraints.
Nonconvex signomial programs have been used in chemical engineering since the 1970's \cite{BW1969-sig-ChemE,DP1973-sigprog-idea,Dembo1978-sig-ChemE}.
More recently, there has been a surge of interest in nonconvex signomial programming for aerospace engineering and transportation systems; see \cite{York2018b-sig-aero,Hall2018-sig-aero,Ozturk2019-sig-aero} for academic work on this topic and \cite{kirschen2021hyperloop} for an industrial example.
The conceptual source of these signomial models in engineering is the simple practice of modeling systems with non-polynomial power laws \cite{gudmundsson2014,Green2019}.

Signomials also provide an avenue for studying sparse polynomials in a degree-independent way.
For example, the nonnegativity certificates studied here have been used to solve high-degree polynomial optimization problems arising in electrical engineering, particularly at scales unmatched by polynomial-specific methods \cite{WKDC-2020}.
Further applications can be found in dynamical systems through analysis of chemical reaction networks and biochemical networks \cite{Mller2015,Mller2019}.
In that context, problems relating to multistationarity can be framed as deciding nonnegativity of a polynomial with degree on the same order as the number of variables \cite{Craciun2005,Pantea2012,Conradi2017}.
A variation of the methods we study have been used to characterize the set of model parameters for which a certain cell signaling mechanism exhibits multistationarity \cite{FKdWY-2020-sonc-multistationarity}.

\subsection{Our starting point for signomial nonnegativity}

Deciding nonnegativity of a signomial $\vct{x} \mapsto \sum_{i,j} M_{ij}\exp(x_i + x_j)$ induced by a square matrix $\mtx{M}$ is equivalent to the well-known matrix copositivity problem.
Since matrix copositivity is NP-hard \cite{Murty1987} it follows that signomial nonnegativity is intractable in the general case.
However, if we are given a nonnegative signomial with at most one negative term, then we can always prove its nonnegativity by either of two arguments.
One such argument appeals to a weighted arithmetic-geometric mean inequality \cite{reznick-1989,Pantea2012,IdW2016}, while the other uses convexity of the exponential function together with convex duality \cite{CS2016}.
The nonnegative signomials with at most one negative term are called \textit{arithmetic-geometric exponentials} or \textit{AGE functions}.
A sum of AGE functions is called \textit{SAGE}.

Using the convex duality argument, the problem of recognizing an AGE function reduces to finding a solution to a single relative entropy inequality.
More generally, the question of whether a signomial with $k$ negative terms is SAGE amounts to finding a solution to $k$ simultaneous relative entropy inequalities \cite{MCW2018}.
These tasks can be accomplished efficiently within the framework of convex \textit{relative entropy programming} (REP), see \cite{CS2016, CS:REP}.
The duality argument also generalizes to constrained problems.
If we have a signomial with at most one negative term, then we can decide its nonnegativity over a convex set $\X$ by finding a feasible solution to a single relative entropy inequality that involves the support function of $\X$ \cite{MCW2019}.
The conceptual approach in \cite{MCW2019} is called \textit{conditional SAGE}.

\subsection{Contributions}\label{subsec:summary_of_results}

To state our results we require some basic definitions.
The symbol $\cA\subset\R^n$ denotes a distinguished finite ground set that contains the origin.
Every vector $\evec\in\cA$ is associated with a ``monomial'' basis function $\e^{\evec}\colon \R^n \to \R$ that takes values $\e^{\evec}(\vct{x}) = \exp\langle\evec,\vct{x}\rangle$.
The resulting \textit{signomial ring} $\R[\cA]$ is the set of all finite products and real-linear combinations of basis functions $\{ \e^{\evec} \}_{\evec\in\cA}$.

Our first result is a signomial \textit{\positiv{}}. I.e., a representation theorem for signomials in $\R[\cA]$ that are positive on a given compact set.
We require two pieces of terminology to state this result.
A signomial is called \textit{$\X$-SAGE} if it can be written as a sum of signomials that are nonnegative on $\X$ and that each have at most one negative term.
A \textit{posynomial} is a signomial with only nonnegative terms.

\begin{positivIntro}[Conditional SAGE \positiv{} (Theorem \ref{thm:newPositiv})]  
    Let $\X$ be a compact convex set and $G$ a finite set of signomials in $\R[\cA]$. If $f \in \R[\cA]$ is positive on $\set{K} \coloneqq \{ \vct{x} \in \X\,:\, g(\vct{x}) \geq 0 \text{ for all } g \in G \}$, then
    there is an $r \in \N$ for which 
    \[
        \textstyle\left(\sum_{\evec\in\cA}\e^{\evec}\right)^r f = \lambda_f + \sum_{g \in G} \lambda_g \cdot g,
    \]
    where $\lambda_f \in \R[\cA]$ is $\X$-SAGE and each $\lambda_g \in \R[\cA]$ is a posynomial.
\end{positivIntro}
Theorem \ref{thm:newPositiv} includes no regularity condition akin to Archimedeanity of $\set{K}$; it is even representation-independent when $G = \emptyset$ (see Corollary \ref{cor:newPositiv}).
It is the first signomial \positiv{} that permits irrational exponents and the first to leverage conditional SAGE in the presence of nonconvex constraints.

Section \ref{sec:lower_bounds} provides a hierarchy of REP relaxations to approach a signomial's minimum from below (see Definition \ref{def:lowerHierarchy}).
Theorem \ref{thm:newPositiv} shows this hierarchy is \textit{complete}.
That is, it can produce bounds of arbitrary accuracy.
Its design includes careful consideration to the concept of ``$\cA$-degree'' (which we study in Section \ref{sec:sig_rings}) and offers improved efficiency and stronger bounds relative to earlier SAGE-based methods. 
We demonstrate our approach with three worked examples: a toy nonconvex quadratic program in five variables, a problem adapted from chemical reaction network theory \cite{Pantea2012,MurrayPhD}, and a chemical reactor design problem \cite{BW1969-sig-ChemE,BW1971}.
The first two of these problems are polynomial in $\vct{t} = \exp\vct{x}$ and so are amenable to the \textit{sums of squares} (SOS) based Lasserre hierarchy \cite{Lasserre2001}; see also \cite{ParriloPhD}.
We find that our methods result in a factor-500 speedup over the Lasserre hierarchy on the second of these problems.
Our method substantially outperforms global solvers SCIP \cite{SCIP:2009,SCIP:2020}, BARON \cite{BARON:1996}, and ANTIGONE \cite{ANTIGONE:2014} on the reactor design problem.

In Section \ref{sec:moment} we develop arbitrarily strong outer approximations of cones of nonnegative signomials (Theorem \ref{thm:outer_rep}), and in Section \ref{sec:upper_bounds} we define hierarchies of convex relaxations for approaching the minimum of a signomial from above (Theorem \ref{thm:generalUB}).
Proving the former theorem requires establishing basic facts on the existence and uniqueness of solutions to signomial moment problems (Propositions \ref{prop:sigRHT} and \ref{prop:unique_rep_meas}).
We state the latter theorem here; let $\R[\cA]_{d}$ denote the space spanned by products of $d$ possibly-nondistinct monomial basis functions $\{\e^{\evec}\}_{\evec\in\cA}$.
\begin{uppersIntro}[Upper Bounds (Theorem \ref{thm:generalUB})]
    Suppose $\operatorname{span}(\cA) = \R^n$, $\set{K}$ is compact, and $\mu$ is a Borel measure with support $\set{K}$. Let $(C_d)_{d \geq 1}$ be a nested sequence of closed convex cones where \emph{(i)} $C_d \subset \R[\cA]_{d}$ contains all posynomials in $\R[\cA]_{d}$, \emph{(ii)} all signomials in $C_d$ are nonnegative on $\set{K}$, and \emph{(iii)} the union $\cup_{d \geq 1} C_d$ includes every signomial in $\R[\cA]$ that is positive on $\set{K}$.
    For any $f \in \R[\cA]$, the values
    \begin{equation*}
        \theta_{d} =  \inf_\psi\left\{ \int f \psi\diff\mu \,:\, \int \psi\diff \mu= 1,\, \psi \in C_d \right\}
    \end{equation*}
    monotonically converge to $\min\{ f(\vct{x}) \,:\, \vct{x} \in \set{K}\}$ from above.
\end{uppersIntro}
The theorem has no assumptions whatsoever on the representation of $\set{K}$.
Moreover, it is agnostic to the precise nature of the convex cones $(C_d)_{d \geq 1}$.
It can be applied with our Theorem \ref{thm:newPositiv}, with earlier signomial \positive{}, and with new signomial \positive{} that are yet to be discovered. 
Section \ref{sec:upper_bounds} provides general guidance on how these hierarchies of upper bounds can be implemented.

In summary, our main theoretical contributions are Theorems \ref{thm:newPositiv}, \ref{thm:outer_rep}, and \ref{thm:generalUB}.
The particular design of our hierarchy from Section \ref{sec:lower_bounds} is a methodological contribution with practical consequences for both signomial and polynomial optimization.
Beyond these specific technical contributions, we believe that our analysis framework of signomial rings creates a foundation for future algebraic studies of signomial nonnegativity and signomial optimization. We offer suggestions for such future studies in Section \ref{sec:outlook}.

\subsection{Related work}\label{subsec:related_work}

Conditional SAGE is a relatively new concept.
Progress has been made in understanding this technique through a convex-combinatorial structural analysis of ``$\X$-SAGE cones'' \cite{MNT2020,NT:sublinear:2021}, group-theoretic dimension reduction techniques \cite{mnrtv:2021}, and a \positiv{} for signomial nonnegativity over compact convex sets \cite{WJYP2020}.
These methods also have demonstrated applications in engineering \cite{WKDC-2020}.

Conditional SAGE generalizes directly from the concept of ordinary SAGE (for global signomial nonnegativity) and so has genealogical connections to nonnegativity certificates based on the arithmetic-geometric mean inequality.
This broader literature includes Reznick's \textit{agiforms} \cite{reznick-1989}, Pantea, Koeppl, and Craciun's characterization of $\R^n_+$-nonnegative circuit polynomials \cite{Pantea2012}, and Iliman and de Wolff's \textit{sums of nonnegative circuits} or \textit{SONC} \cite{IdW2016}.
The nature of the SAGE-SONC relationship was resolved in \cite[\S 5]{MCW2018} with the introduction of \textit{SAGE polynomials}.
One can also understand the SAGE-SONC relationship by implicitly reading results by Wang \cite{Wang2018a}.
The unifying perspective is that both of these approaches characterize elementary nonnegative functions that are a sum of ``monomials'' where at most one monomial contributes a negative value to the overall sum.
Further generalizations along these lines can be found in \cite{KNT2019} (which allows monomials $\prod_{i=1}^n |x_i|^{\eentry_i}$) and \cite[\S 4]{MCW2019}.

Much of the interest in arithmetic-geometric based nonnegativity certificates stems from their sparsity preservation properties \cite{Wang2018b,MCW2018,KNT2019}.
As a consequence of this sparsity preservation, \textit{it is possible to implement these methods with complexity that depends only on the number of terms in the signomial or polynomial} \cite{MCW2018}.
In the polynomial setting, much work has gone into the development of variants of SOS that are capable of exploiting \textit{structured sparsity}, see, e.g., \cite{KKW-2004-sparsesos-supports,WKKM-2006-sparsesos-chordal,WLX2019-sparsesos-crosssparse,WML2021-TSSOS}.
Arithmetic-geometric certificates are significant in the polynomial literature largely because they can take advantage of \textit{unstructured sparsity}.

To our knowledge, the earliest signomial \positiv{} is Delzell's extension of the weak form of P{\'{o}}lya's \positiv{} to signomials with rational exponents.
Chandrasekaran and Shah gave two \positive{} for ordinary SAGE in \cite{CS2016}: one for $\set{K}=\R^n$ and one for Archimedean $\set{K}$.
Wang et al.\ developed the first conditional SAGE \positiv{} in the case when $\set{K}$ is a compact convex set \cite{WJYP2020}.
We compare our Theorem \ref{thm:newPositiv} to these SAGE-based \positive{} in Section \ref{subsec:compare_positiv}.
The proof of our result relies on a reduction to the Dickinson-Povh \positiv{} for homogeneous polynomials with infinitely many homogeneous polynomial inequality constraints \cite{DP2014}.
Dickinson and Povh have subsequently used their \positiv{} to develop complete hierarchies for \textit{polynomial cone programming} \cite{DP2019}.

Our results in signomial moment theory are manifestations of more abstract theorems in the functional analysis literature.
We particularly rely on a generalized Riesz-Haviland Theorem proven by Marshall almost two decades ago \cite{marshall2003}.
More recent work along these lines includes \cite{GIKM:2018,IKM:2018} and especially \cite{curto2020truncated} -- which may have computational implications for our hierarchy of lower bounds.
Moment problems involving univariate power functions have been studied in \cite{NIT-2003,NIPPT-2005} for purposes of estimating probability density functions.

The idea for our outer-approximations of signomial nonnegativity cones and our hierarchy of upper bounds is adapted from a work by Lasserre \cite{Lasserre2011} and inspired by a generalization thereof by de Klerk, Lasserre, Laurent, and Sun \cite{dKLLS2017}.
Several investigations have been conducted to determine rates of convergence for these upper bounds under various conditions \cite{dKLS2017-convergence,dKHL2017-improved,dKL2020-convergence,dKL2020-worst,SL2020-convergence-improved}.
We leave questions of convergence rates with our method to future work.

\subsection{Conventions}\label{subsec:basic_notation}

The index set from $1$ to $n$ is denoted by $[n]$.
We use $\R_+$ for the nonnegative reals and $\R_{++}$ for the positive reals.
The zero vector and vector of all ones (in spaces clear from context) are given by $\vct{0}$ and $\vct{1}$.
For sets $A, B$, we write $A \subset B$  for non-strict inclusion.
The exponential function is extended first to vectors in an elementwise fashion and then to sets in a pointwise fashion.
All logarithms are base $e = \exp(1)$.
Domains of signomials or polynomials (or subsets of such domains) are written with capital letters in sans-serif font.
For such a domain $\set{K}$, we sometimes use ``$\set{K}$-nonnegative'' as an adjective to indicate that a function is nonnegative on $\set{K}$.
We similarly use the term ``$\set{K}$-positive.''

\section{Preliminaries on conditional SAGE}\label{sec:sage_background}

This section reviews the machinery of conditional SAGE and introduces basic notation used throughout the article.

For a finite set $\genExps \subset \R^n$, we use $\R^{\genExps}$ to denote the real $|\genExps|$-tuples indexed by $\evec\in \genExps$.
With this notation we can identify a signomial $\sum_{\evec\in \genExps}c_{\evec}\e^{\evec}$ by its coefficient vector $\vct{c}\in\R^{\genExps}$.
Given some $\bevec\in \genExps$, the vector $\vct{c}_{\setminus\bevec}$ is formed by dropping $c_{\bevec}$ from $\vct{c}$.
We use $\{\vct{\delta}_{\evec}\}_{\evec\in \genExps}$ for the standard basis in $\R^{\genExps}$ and we identify $\genExps$ with a linear operator $\genExps \colon \R^n \to \R^{\genExps}$ that takes values $\genExps\vct{x} = (\langle \evec, \vct{x}\rangle)_{\evec \in \genExps}$.
When thinking of $\genExps$ as a linear operator, we can effectively regard it as a matrix with rows sorted according to the lexicographic order on $\R^n$.

\subsection{The primal perspective}\label{subsec:prim_sage}

Recall from Section \ref{subsec:summary_of_results} that we call a signomial \textit{$\X$-SAGE} if it can be written as a sum of $\X$-nonnegative signomials each with at most one negative term.
Unsurprisingly, we refer to a signomial as \textit{$\X$-AGE} if it is nonnegative on $\X$ and has at most one negative term.
We now put our notation to work in the following characterization of $\X$-AGE functions.
\begin{proposition}[\cite{MCW2019}]\label{prop:represent_age_cones}  
     Let $\X$ be a convex set in $\R^n$ and $\genExps\subset \R^n$ be finite.
     For $\bevec \in \genExps$ and $\vct{c}\in \R^{\genExps}$ with $\vct{c}_{\setminus\bevec} \geq \vct{0}$, the signomial
     $\sum_{\evec\in \genExps}c_{\evec}\e^{\evec}$ is $\X$-AGE if and only if
    some $\vct{\nu} \in \R^{\genExps}$ satisfies the conditions $\langle\vct{1},\vct{\nu}\rangle = 0$,
    \begin{equation}\label{eq:represent_age_cones}
        \sup_{\vct{x}\in \X}\left\langle \vct{\nu}, -\genExps \vct{x} \right\rangle + \sum_{\evec\in \genExps\setminus\bevec} \nu_{\evec}\log\left(\frac{\nu_{\evec}}{c_{\evec}}\right) + \nu_{\bevec} \leq c_{\bevec},
    \end{equation}
    and $\vct{\nu}_{\setminus\bevec} \geq \vct{0}$. 
\end{proposition}
An individual term $u\log(u/v)$ in equation \eqref{eq:represent_age_cones} is the \textit{relative entropy} between the scalars $(u, v)$.
This function is continuously extended to allow $u = 0$ or $v = 0$ and is convex when viewed as a function on $\R_+^2$.
The other term appearing in \eqref{eq:represent_age_cones} is the \textit{support function} of the convex set $-\genExps\X$ and is likewise convex in $\vct{\nu}$.

It is often necessary (e.g., when performing computations) to work with finite dimensional cones of $\X$-SAGE functions.
For this we introduce cones of $\X$-SAGE signomials supported on $\genExps$:
\begin{equation}\label{eq:def_x_sage}
    \csage_{\X}(\genExps) := \left\{ f \,:\, f = \sum_{\evec\in \genExps}c_{\evec}\e^{\evec} \text{ is } \X\text{-SAGE} \right\}.  
\end{equation}
To represent such a cone we rely on a crucial sparsity-preservation property: if $f$ is supported on $\genExps$ and has $k \geq 1$ negative coefficients, then $f$ is $\X$-SAGE if and only if it can be written as a sum of $k$ $\X$-AGE functions, each supported on $\genExps$ \cite{MCW2019}.
A full $\X$-SAGE cone can therefore be expressed as a Minkowski sum
\begin{equation*}
    \csage_{\X}(\genExps) = \textstyle\sum_{\bevec \in \genExps} \csage_{\X}(\genExps,\bevec)
\end{equation*}
of simpler cones
    $\csage_{\X}(\genExps,\bevec) := \left\{ f \,:\, f = \textstyle \sum_{\evec\in \genExps} c_{\evec}\e^{\evec} \text{ is } \X\text{-nonnegative}, ~ \vct{c}_{\setminus \bevec} \geq \vct{0} \right\}$.
By Proposition \ref{prop:represent_age_cones}, each of these $|\genExps|$ summands admits an explicit representation that is jointly convex in a signomial's coefficient vector and an auxiliary variable of size $|\genExps|$.
This implies a worst-case representation for $\csage_{\X}(\genExps)$ in $|\genExps|$ relative entropy and support function inequalities and $2|\genExps|^2$ decision variables.

\subsection{The dual perspective}\label{subsec:dual_sage}

Thinking of $\csage_{\X}(\genExps)$ and $\csage_{\X}(\genExps,\bevec)$ as cones of functions follows the convention from the polynomial optimization literature.
However, since vectors $\vct{c} \in \R^{\genExps}$ are in correspondence with signomials $\sum_{\evec\in \genExps}c_{\evec}\e^{\evec}$, we are free to regard these sets as cones in $\R^{\genExps}$.
Thinking of these sets as cones of coefficients makes it easier to derive the corresponding dual cones.
Recall, the dual to a convex cone $C \subset \R^m$ is $C^\dagger = \{ \vct{v} \,:\, \langle\vct{u},\vct{v} \rangle \geq 0 \text{ for all } \vct{u} \in C\}$.
For each $\bevec \in \genExps$, we find
\begin{align}\label{eq:represent_dual_x_age}
    \csage_{\X}(\genExps,\bevec)^\dagger = \cl\Big\{& \vct{v} \in \R^{\genExps}_{++} \,:\, \text{some } \vct{z} \text{ satisfies } \vct{z} / v_{\bevec} \in \X \\
    & \text{and } v_{\bevec}\log\left(\frac{v_{\evec}}{v_{\bevec}}\right) \geq \langle\evec - \bevec, \vct{z}\rangle \, \forall \, \evec \in \genExps \Big\}. \nonumber
\end{align}
Standard rules in conic duality then tell us that $\csage_{\X}(\genExps)^\dagger = \cap_{\bevec \in \genExps} \csage_{\X}(\genExps,\bevec)^\dagger$.

The dual formulation facilitates solution recovery when using SAGE-based REP relaxations for signomial optimization (see \cite[\S 3.2]{MCW2019}).
The dual formulation also provides a window into how the complexity of $\X$ affects the complexity of an $\X$-SAGE cone.
If $\X = \{ \vct{x} \in \R^n\,:\, \mtx{G}\vct{x} + \vct{h} \in C\}$ for a matrix $\mtx{G}$, vector $\vct{h}$, and convex cone $C$, then the constraint $\vct{z} / v_{\bevec} \in \X$ can be written as $\mtx{G} \vct{z} + \vct{h}v_{\bevec} \in C$.
Therefore the problem of optimizing over a dual $\X$-SAGE cone is not much harder than optimizing over a dual $\R^n$-SAGE cone (for which efficient algorithms exist) and optimizing over $\X$ (which can vary in difficulty).

\section{Signomial rings}\label{sec:sig_rings}

As we explained in Section \ref{subsec:summary_of_results}, $\cA \subset \R^n$ is a finite set that contains the origin and the signomial ring $\struc{\R[\cA]}$ is the $\R$-algebra generated by the basis functions $\{ \e^{\evec} \}_{\evec\in\cA}$.
This section explores a way of grading signomial rings by degree.
We begin by defining a sequence of sets
\[
   \cA_{d} = \left\{ \sum_{\evec\in\cA} w_{\evec} \evec \,:\, \vct{w} \in \N^{\cA},~ \langle\vct{1}, \vct{w}\rangle \leq d \right\} \quad\text{ for }\quad d \geq 1.
\]
Where we note that $\cA_1 = \cA$.
Next, we formally define the \textit{support} of a signomial $f$, denoted $\supp(f)$, as the smallest set $A \subset \R^n$ for which $f \in \spann\{\e^{\evec}\}_{\evec\in A}$.
The \textit{$\cA$-degree} of $f$ is then the smallest integer $d$ for which $\supp(f) \subset \cA_{d}$, and this number is denoted $\deg_{\cA}(f)$.
We use $\R[\cA]_{d}$ for the space of signomials of $\cA$-degree at most $d$.

The definition of $\cA$-degree is, by itself, enough to get through the proof of our \positiv{} in Section \ref{sec:sig_positiv}.
In later sections it is important to understand the properties of $\cA$-degree.
We explore those basic properties here.

\subsection{The $\cA$-degree of a single signomial}

The concept of $\cA$-degree is artificially imposed on signomials.
If $\supp(f) \subset \cA$, then the $\cA$-degree of $f$ is trivially one.
Note that unless $\cA$ is decided by some external factor, one can always update $\cA \leftarrow \cA \cup \supp(f)$, and so every signomial has degree one when considered in a suitable ring.
In fact, if we chose to interpret previous SAGE-based hierarchies in terms of signomial rings, then we find that they always make such a choice for $\cA$.

In this article we show it can be advantageous to consider signomials in rings where their resulting $\cA$-degree is greater than one.
This creates a need to determine $\deg_{\cA}(f)$ when there is no special relationship between $\supp(f)$ and $\cA$.
The naive thing to do in this case is to explicitly construct the sets $\cA_2, \cA_3, \ldots$ and return the first $d$ where $\supp(f) \subset \cA_{d}$.
However, that algorithm does not terminate if $f$ does not belong to $\R[\cA]$.
In practice we suggest $\cA$-degree and membership in $\R[\cA]$ be determined by a separable calculation $\deg_{\cA}(f) = \max\{ \deg_{\cA}(\e^{\evec}) \,:\, \evec \in \supp(f) \}$ involving simple integer-linear programs
\begin{align*}
    \deg_{\cA}(\e^{\evec}) = \inf\{\, \ell \, : \, & \ell \geq 1 \text{ and } \vct{w} \in \N^{\cA} \text{ satisfy}\\ &\langle\vct{1},\vct{w}\rangle \leq \ell \text{ and } {\textstyle\sum_{\bevec\in\cA} \bevec w_{\bevec} = \evec} \}.
\end{align*}
We also propose that these values are memoized in a suitable data structure to amortize the cost of computing them.
Such memoization can lead to significant speed-ups when computing the $\cA$-degrees of many signomials at once.

The following example shows how updating $\cA$ can have both large and small effects on a given signomial's $\cA$-degree.

\begin{example}\label{ex:consider_different_rings}
    Suppose $\mtx{M} \in \R^{n \times n}$ is a dense symmetric matrix and consider the polynomial $p(\vct{t}) = \prod_{i=1}^n t_i + \sum_{i,j=1}^n M_{ij} t_i t_j$.
    From $p$ we construct the signomial  $f(\vct{x}) \coloneqq p(\exp \vct{x})$. 
    When $f$ is viewed in the rings generated by
    \[
    \cA = \{\vct{0}\} \cup \{\vct{\delta}_{i}\}_{i=1}^n, ~~ \cA' = \cA \cup \{ \vct{1} \}, ~\text{ and }~ \cA'' = \cA' \cup \{ \vct{\delta}_i + \vct{\delta}_j \,:\, (i,j) \in [n]^2 \}
    \]
    we have $\deg_{\cA}(f) = n$, $\deg_{\cA'}(f) = 2$, and $\deg_{\cA''}(f) = 1$ respectively. \hfill $\blacklozenge$
\end{example}

Although $\cA$-degree is not intrinsic to signomials, it exhibits essential properties of coordinate-system invariance.
For any $\vct{b}$ in $\R^n$ and $f \in \R[\cA]$, the signomial $g(\vct{x}) = f(\vct{x} - \vct{b})$ has $\deg_{\cA}(g) = \deg_{\cA}(f)$.
This shift invariance becomes valuable when we discuss numerical optimization in Section~\ref{sec:lower_bounds}.
In addition, for any nonsingular matrix $\mtx{B} \in \R^{n \times n}$, the signomial $g$ defined by $g(\vct{x}) = f(\mtx{B}\vct{x})$ has $\deg_{[\cA\mtx{B}]}(g) = \deg_{\cA}(f)$.
These invariants are reflected in our proof techniques in Sections~\ref{sec:sig_positiv} and \ref{sec:sig_moments}, which are unaffected by changes to $(\cA, \X)$ that preserve the linear image $\cA\X \subset \R^{\cA}$ up to a translation in the range of $\cA$.

\subsection{Behavior of $\cA$-degree under multiplication}\label{subsec:submult_Adeg}

Polynomial rings enjoy a property where given two nonzero polynomials $p$ and $q$, the degree of the product $pq$ is the sum of degrees $\deg p$ and $\deg q$.
Signomial $\cA$-degree partly preserves this property.
For any two signomials $f, g$ in a common ring $\R[\cA]$, we have
\begin{equation}\label{eq:submult_Adeg}
    \deg_{\cA}(f g) \leq \deg_{\cA}(f) + \deg_{\cA}(g).
\end{equation}
However, the inequality in \eqref{eq:submult_Adeg} can be strict even when both $f$ and $g$ are nonzero.
A trivial example of strict inequality is given by $f = \e^{\vct{0}}$, which satisfies $f = f^2$ and $\deg_{\cA}(f) = 1$.
Here is a nontrivial example.

\begin{example}\label{ex:drop_Adeg}
    Consider an integer $k \geq 3$ and
    $\cA = \{0, 1, k\}$.
    Then $f(x) = \exp(x)$ has $\deg_{\cA}(f^{p}) = p$ for $1 \leq p < k$, and yet $\deg_{\cA}(f^k) = 1$. \hfill $\blacklozenge$
\end{example}

The potential for strict inequality in \eqref{eq:submult_Adeg} complicates the process of grading $\R[\cA]$ by $\cA$-degree.
However, this complication can actually be used to our advantage.
The idea is that for a nonconstant polynomial $p$, the only polynomial $q$ for which $\deg(pq) < \deg(p)$ is $q = 0$.
By contrast, there are certain support sets $\cA$, signomials $f \in \R[\cA]$, and nontrivial linear subspaces $S \subset \R[\cA]$ where $\deg_{\cA}(f g) < \deg_{\cA}(f)$ for every signomial $g \in S$.

\begin{example}\label{ex:multiplier_subspace}
    Consider $\cA = \{-1, 0, 1, 2\}$ and the signomial $f(x) = \exp(3x)$.
    Clearly $\deg_{\cA}(f) = 2$, and yet $\deg_{\cA}(g f) \leq 1$ for every signomial $g$ in
    the one-dimensional linear space $S = \{ c \exp(-x) \,:\, c \in \R\} \subset \R[\cA]$. \hfill $\blacklozenge$
\end{example}

In view of Example~\ref{ex:multiplier_subspace}, we have a need to take a signomial $f$ and describe the inclusion-maximum $A \subset \cA_{d}$ where $\deg_{\cA}(f g) \leq d$ for every signomial $g$ supported on $A$.
We denote this set by $\invsupp_{d}(f)$ and note that it can be expressed as
\begin{equation}\label{eq:def_invsupp}
    \invsupp_{d}(f) = \{ \evec \in \cA_d \,:\, \evec + \supp(f) \subset \cA_d \}.
\end{equation}
In terms of these support sets, inequality~\eqref{eq:submult_Adeg} simply tells us that if $f$ is of an $\cA$-degree $k$ strictly smaller than $d$, then $\invsupp_{d}(f)$ contains $\cA_{d-k}$.


Of course, $\cA$-degree can behave like polynomial degree in certain situations.
Here is one prominent case.

\begin{proposition}\label{prop:degree_behavior}
    Suppose $f$ is a nonconstant signomial in $\R[\cA]$.
    If all extreme points of the convex hull of $\cA_r$ are among the support of a signomial $g \in \R[\cA]_{r}$, then $\deg_{\cA}(g f) = r + \deg_{\cA}(f)$.
    In particular, the $\cA$-degree of $(\sum_{\evec\in\cA}\e^{\evec})^r f$ is equal to $r + \deg_{\cA}(f)$ whenever $f$ is nonconstant.
\end{proposition}

\section{A Positivstellensatz}\label{sec:sig_positiv}

Throughout this section $f$ is a signomial in $\R[\cA]$, $\X$ is a compact convex subset of $\R^n$, and $G \subset \R[\cA]$ is finite.
Here we present a characterization of signomials that are positive on sets
\[
    \set{K} = \{\vct{x} \in \X \,:\, g(\vct{x}) \geq 0 \text{ for all } g \text{ in } G \}.
\]
Note that such sets are nonconvex in general, because there are signomials that are not concave.
One example of a nonconvex signomial constraint is $\sum_{\evec\in A}\e^{\evec}(\vct{x}) - 1 \geq 0$ where $A$ contains two nonzero vectors in $\R^n$.

In later sections, the following result will be used to develop hierarchies of successively stronger convex relaxations for approaching $f_{\set{K}}^\star = \inf_{\vct{x}\in \set{K}} f(\vct{x})$ from below and above (\S \ref{sec:lower_bounds} and \S \ref{sec:upper_bounds} respectively).

\begin{theorem}\label{thm:newPositiv}
    If $f$ is positive on $\set{K}$, then there exists an $r \in \N$ for which 
    \begin{equation}\label{eq:positivDecompClaim}
        \textstyle\left(\sum_{\evec\in\cA}\e^{\evec}\right)^r f = \lambda_f + \sum_{g \in G} \lambda_g \cdot g,
    \end{equation}
    where $\lambda_f \in \R[\cA]$ is $\X$-SAGE and the $\lambda_g \in \R[\cA]$ are posynomials.
\end{theorem}

Note how the theorem requires $f$ to be \textit{positive} on $\set{K}$ but only guarantees an identity that implies \textit{nonnegativity} on $\set{K}$.
The gap between $\set{K}$-positive signomials and $\set{K}$-nonnegative signomials is important in optimization, as it makes the difference between finite versus asymptotic convergence of our lower bounds.
To improve one's chances of finding an identity like \eqref{eq:positivDecompClaim} when $f_{\set{K}}^\star = 0$, the multipliers $(\lambda_g)_{g \in G}$ can be taken as $\X$-SAGE signomials rather than merely posynomials (we make this change in Section \ref{sec:lower_bounds}).
There is also an interesting case when $G = \emptyset$, where the representation from Theorem~\ref{thm:newPositiv} uses no multipliers whatsoever.

\begin{corollary}\label{cor:newPositiv}
    If $f$ is positive on $\X$, then there exists a natural number $r$ where the signomial $(\sum_{\evec \in \cA} \e^{\evec})^r f$ is $\X$-SAGE.
\end{corollary}

Our proof of Theorem~\ref{thm:newPositiv} is presented in Section~\ref{subsec:proof_positiv}; it relies on two black-box lemmas, which are proven in Sections~\ref{subsec:proof_lem_dp2sage} and \ref{subsec:proof_lem_polynomialize_X}.
The second of these lemmas contains our main technical innovation outside the use of signomial rings, and we provide some extra commentary on the lemma following its proof (see Remark~\ref{rem:discuss_QX_ideal}).

\subsection{Comparison to existing \positive{}}\label{subsec:compare_positiv}

We begin by paraphrasing two existing SAGE \positive{} in the language of signomial rings.
The first such \positiv{} was proven in \cite{CS2016} when the concept of SAGE certificates was introduced.
To state the result we use
\begin{equation}\label{eq:def_RqG}
    R_q(G) \coloneq \left\{ \textstyle\prod_{i=1}^q g_i \,:\, g_1,\ldots,g_q \in \{1\}\cup G\right\} 
\end{equation}
 to denote the set of all products of $q$ (possibly nondistinct) signomials from $G \cup \{1\}$.
 
\begin{theorem}[\cite{CS2016}]\label{prop:ord_sage_constr_positiv}  
    Suppose $\{f\}\cup G \subset \R[\cA]$ for exponents $\emat \subset \Q^n$.
    Further, assume that $G$ explicitly includes signomials $\{U - \e^{\evec},\e^{\evec} - L\}_{\evec \in \cA}$ for some positive constants $U, L$, so that $\X = \{ \vct{x} \,:\, U \geq \e^{\evec}(\vct{x}) \geq L \, \forall \, \evec \in \cA \}$ is compact.
    If $f$ is positive on $\set{K}$, then there exists a natural number $q$ and a family of $\R^n$-SAGE signomials $( \lambda_h )_{h \in R_q(G)} \subset \R[\cA]$ that satisfy 
        $f =  \textstyle\sum_{h \in R_q(G)} \lambda_h \cdot h$.
\end{theorem}

Theorem~\ref{prop:ord_sage_constr_positiv} does not involve conditional SAGE.
We have simply phrased it to emphasize that if $(f,G)$ satisfy its hypothesis, then $(f,G,\X)$ satisfy the hypothesis of Theorem~\ref{thm:newPositiv} for the indicated choice of $\X$.
Our Theorem~\ref{thm:newPositiv} is qualitatively different from the earlier Theorem~\ref{prop:ord_sage_constr_positiv} in that the former does not require taking products of constraint functions.
This distinction is of practical importance.

The next \positiv{} was proven by Wang et al.~\cite{WJYP2020} shortly after the introduction of conditional SAGE certificates.
Its scope is limited to problems where $G = \emptyset$ (i.e., $\set{K}=\X$), but is nevertheless distinguished in how its conclusion is independent of the representation of $\X$.

\begin{theorem}[\cite{WJYP2020}]\label{prop:first_cond_sage_positiv}  
    Suppose the exponents $\evec \in \cA$ are rational and that $f$ has $\cA$-degree one.
    If $f$ is positive on $\X$, then there exists a natural number $r$ for which $(\sum_{\evec\in\cA}\e^{\evec})^r f$ is $\X$-SAGE.
\end{theorem}

Our Theorem~\ref{thm:newPositiv} naturally generalizes Wang at al.'s Theorem~\ref{prop:first_cond_sage_positiv} to the constrained setting, and in fact our proof of Theorem \ref{thm:newPositiv} draws much inspiration from \cite{WJYP2020}.
The comparison between Wang et al.'s Theorem~\ref{prop:first_cond_sage_positiv} and our Corollary~\ref{cor:newPositiv} is best illustrated with an example.

\begin{example}\label{ex:our_cor_vs_wjyp}
    Return to the signomials from Example~\ref{ex:consider_different_rings}.
    Let $M_{ij}$ denote the entries of the matrix $\mtx{M}$ so the signomial
    $f = \e^{\vct{1}} + \textstyle\sum_{i,j=1}^n M_{ij}\e^{[\vct{\delta}_i + \vct{\delta}_j]}$
    is positive on $\X$.
    If we want a certificate that $f$ is nonnegative over $\X$, then Theorem~\ref{prop:first_cond_sage_positiv} says it suffices to look for $\X$-SAGE decompositions of functions
    $\mathcal{L}_r = \left(\e^{\vct{0}} + \e^{\vct{1}} + {\textstyle\sum_{i\leq j}^n \e^{[\vct{\delta}_i + \vct{\delta}_j]}}\right)^r f$.
    Since the number of terms in $\mathcal{L}_r$ grows as $O(n^{2r})$, the sizes of the REPs used when searching for the $\X$-SAGE decompositions can scale as rapidly as $O(n^{4r})$.
    By contrast, Corollary \ref{cor:newPositiv} says it suffices to look for $\X$-SAGE decompositions of functions
    $\mathcal{L}_r' = \left(\sum_{\evec\in\cA}\e^{\evec}\right)^r f$ where $\cA$ is \textit{any} set for which $f$ belongs to $\R[\cA]$.
    In particular, we can use $\cA = \{\vct{0}\}\cup\{\vct{\delta}_i\}_{i=1}^n$, so the number of terms in $\mathcal{L}_r'$ would grow as only $O(n^r) \ll O(n^{2r})$.
    Corollary \ref{cor:newPositiv} therefore justifies a whole family of convergent REP relaxation hierarchies with different efficiency profiles as the hierarchy parameter increases. \hfill $\blacklozenge$
\end{example}

    Besides the comparisons we have drawn so far, we make no requirement that the exponents $\cA$ are rational.
    The distinction between rational and irrational exponents has some significance.
    In 2008, Delzell studied the extent to which P{\'{o}}lya's theorem (for homogeneous polynomials positive on the simplex) generalizes to signomials in $\R[\R^n]$ \cite{Delzell2008}.
    Using the convention of signomials as functions $\vct{t} \mapsto \sum_{\evec} c_{\evec}\vct{t}^{\evec}$, \cite{Delzell2008} showed that the bivariate signomial $f(\vct{t}) = t_1^{2} + t_2^2 - t_1^{1+\epsilon}t_2^{1-\epsilon}$ is positive on $\R^2_{++}$ when $\epsilon \in (-1,1)$, and yet when $\epsilon$ is irrational, there exists no ``homogeneous'' signomial $g \in \R[\R^n]$ for which $g f$ has nonnegative coefficients.
    That is, it is impossible to obtain a P{\'{o}}lya-like \positiv{} to certify global nonnegativity of a signomial with general irrational exponents.
    Our results show that under a different model of signomial rings and a compactness assumption, $\X$-SAGE certificates characterize signomials positive on $\X$ even when the exponents are irrational.

\subsection{Proof of Theorem \ref{thm:newPositiv}}\label{subsec:proof_positiv}

Our proof works by mapping a signomial problem to a polynomial problem, applying a polynomial \positiv{}, and then mapping back to signomials.
As a first step long this path, we shall call a polynomial $p$ a \textit{\struc{\polation{}}} of $f$ if $f(\vct{x}) = p(\Exp \cA \vct{x})$ for all $\vct{x}$ in $\R^n$.
Note that every signomial $f$ has a homogeneous \polation{} of degree $\deg_{\cA}(f)$ (since $\vct{0} \in \cA$).
Henceforth, we assume all \polation{}s are homogeneous.

\begin{example}\label{ex:polynomialize}
    Consider the univariate case $\cA = [1/4;\, 1/2;\, 1/3;\, 0] \in \R^{4 \times 1}$.
    The signomial $f(x) = \exp(x)$ admits several \polation{}s, among them $p_1(\vct{y}) = y_1^2 y_2$ and $p_2(\vct{y}) = y_3^3$.
    Note that $p_1 \cong p_2$ on the variety $\{ \vct{y} \in \R^4 \,:\, y_1^2 = y_2, y_2^2 = y_3^3\}$.  
    \hfill $\blacklozenge$
\end{example}

The above example suggests a signomial ring $\R[\cA]$ is equivalent to the ring of polynomials on $\R^{\cA}$, modulo a suitable binomial ideal to capture the relationships between $\evec,\bevec\in \cA$.
In order to interpret a signomial ring in this way, we need $\Exp(\cA\R^n)$ to be the intersection of a toric variety with a positive orthant.
By considering $\cA = \{ 1, \sqrt{2}, 0\}$ we see that this cannot be the case in general.
This provides one example of how fully general signomial rings are resistant to techniques from traditional algebraic geometry.
However, the differences between signomial and polynomial rings are less pronounced when considering these functions only over compact sets.
Specifically, by restricting our attention to signomial nonnegativity on compact sets, we are able to prove Theorem~\ref{thm:newPositiv} by appeal to the following results of Dickinson and Povh.

\begin{theorem}[\cite{DP2014}]\label{thm:complete:dpFinite}  
    Let $p$ be a homogeneous polynomial on $\R^{\cA}$, and let $Q$ be a finite set of homogeneous polynomials on $\R^{\cA}$ that includes the constant polynomial $\vct{y} \mapsto 1$.
    If $p$ is positive on $\{ \vct{y} \in \R^{\cA}_+: q(\vct{y}) \geq 0 \text{ for all } q \in Q\}\setminus\{\vct{0}\} $,
    then for some $r \in \N$ there exist homogeneous polynomials $\{ \mu_q\}_{q \in Q}$ with nonnegative coefficients such that $(\sum_{\evec \in \cA} y_{\evec})^r p(\vct{y}) = \sum_{q\in Q} \mu_q(\vct{y}) q(\vct{y})$. 
\end{theorem}

For general choices of $(\cA, \X)$ we also require a reduction from a semi-infinite nonnegativity problem to a finite nonnegativity problem, as follows.

\begin{theorem}[\cite{DP2014}]\label{thm:dpReduction}  
    Consider a countable set $\{ p \} \cup Q$ of homogeneous polynomials on $\R^\cA$.
    If $p$ is positive on $\{ \vct{y} \in \R^{\cA}_+ \,:\, q(\vct{y}) \geq 0 \, \text{ for all } \, q \in Q\}\setminus\{\vct{0}\}$, then there exists a finite $Q' \subset Q$ for which $p$ is positive on $\{ \vct{y} \in \R^{\cA}_+ \,:\, q(\vct{y}) \geq 0 \, \forall \, q \in Q' \}\setminus\{\vct{0}\} $.
\end{theorem}

Next, given a polynomial $p$ on $\R^{\cA}$, we have the \textit{\struc{\sigation{}}} $\vct{x} \mapsto p(\Exp\cA\vct{x})$.
\Sigation{} transparently preserves important algebraic properties.
For example, if we signomialize a polynomial that has nonnegative coefficients in the monomial basis, then we obtain a posynomial.
In addition, if $g$ is the \sigation{} of a polynomial $p$ and $p$ is a \polation{} of some signomial $f$, then $g = f$.
The following lemma roughly shows how these concepts help map Dickinson-Povh certificates to conditional SAGE certificates.

\begin{lemma}\label{lem:dp2sage}
    Let $Q = Q_1 \cup Q_2$ be a finite set of polynomials on $\R^{\cA}$ where the \sigation{} of each $q \in Q_2$ is $\X$-AGE and suppose $p$ is a \polation{} of $f$.
    If 
        $\textstyle{(\sum_{\evec \in \cA} y_{\evec})}^r p(\vct{y}) = \sum_{q \in Q} \mu_q(\vct{y}) q(\vct{y})$
    for polynomials $\mu_q$ with nonnegative coefficients and a natural number $r$,
    then there exists an $\X$-SAGE function $\lambda_f \in \R[\cA]$ for which
        $\textstyle{(\sum_{\evec\in\cA}\e^{\evec})^r}f = \lambda_f + \sum_{q\in Q_1} \lambda_q g_q$,
    where $\lambda_q$ is the \sigation{} of $\mu_q$ and $g_q$ is the \sigation{} of $q$.
\end{lemma}

The work in our proof of Theorem \ref{thm:newPositiv} is to derive polynomial data from signomial data so that the hypotheses of Lemma \ref{lem:dp2sage} are satisfied.
Much of this work is accomplished in our next lemma.

\begin{lemma}\label{lem:polynomialize_X}
    There exists a countable set of homogeneous polynomials $Q(\X)$ on $\R^\cA$ satisfying the following properties:
    \begin{itemize}
        \item[(i)] each $q \in Q(\X)$ has at most two terms,
        \item[(ii)]  $\exp\cA\X = \{\vct{y} \in \R^{\cA} \,:\, q(\vct{y}) \geq 0 \, \, \text{ for all } \,\,q \in Q(\X),~ y_{\vct{0}} = 1\}$,
        \item[(iii)] if $\vct{y}$ is a nonzero vector where $q(\vct{y}) \geq 0$ for all $q \in Q(\X)$, then $\vct{y} > \vct{0}$.
    \end{itemize}
\end{lemma}
As a consequence of conditions (i) and (ii) in the lemma, the \sigation{} of any $q \in Q(\X)$ has at most two terms and is $\X$-nonnegative.

\begin{proof}[Proof of Theorem \ref{thm:newPositiv}]
    Fix $f > 0$ on $\set{K} := \{ \vct{x} \in \X\,:\, g(\vct{x}) \geq 0 \, \forall \, g \in G \}$.
    Let $p$ be a \polation{} of $f$, $Q(G)$ be a set of \polation{}s of $G$ (one polynomial for each signomial in $G$), and $Q(\X)$ be as in Lemma \ref{lem:polynomialize_X}.
    Define the region $\set{K}_p = \exp\cA\set{K}$ within $\R^{\cA}$ and the set of polynomials $Q = Q(\X) \cup Q(G)$.
    
    We begin by noting how $\set{K}_p = \left(\exp\cA\X\right) \cap \{ \exp\cA\vct{x} \,:\, \vct{x}\in \R^n,~ g(\vct{x}) \geq 0\,\forall\,g \in G\}$.
    Next, we apply Lemma \ref{lem:polynomialize_X} and we use the fact that $g(\vct{x}) = q(\exp\cA\vct{x})$ when $q$ is a \polation{} of $g$.
    This allows us to write $\set{K}_p$ purely in terms of homogeneous polynomials: $\set{K}_p = \{ \vct{y} \,:\, y_{\vct{0}} = 1,~ q(\vct{y}) \geq 0 \, \forall \, q \in Q \}$.
    From here we drop the constraint $y_{\vct{0}} = 1$ to obtain $\set{T} = \{ \vct{y}\,:\, q(\vct{y}) \geq 0 \, \forall \, q \in Q \}$.
    Apply the third property of $Q(\X)$ from Lemma~\ref{lem:polynomialize_X} to see that $\set{T}\setminus\{\vct{0}\}$ is contained within $\R^{\cA}_{++}$.
    
    Let $d = \deg_{\cA}(f)$ and consider an arbitrary vector $\vct{y} \in \set{T}\setminus\{\vct{0}\}$.
    Since $p$ is homogeneous, we have $p(\vct{y}) = y_{\vct{0}}^d p(\vct{y}/y_{\vct{0}})$.
    Similarly, because all polynomials defining $\set{T}$ are homogeneous, we have that
    $\tilde{\vct{y}} := \vct{y} / y_{\vct{0}}$ is in $\set{K}_p$.
    By the definition of $\set{K}_p$ we know that every vector $\tilde{\vct{y}}\in \set{K}_p$ can be represented as $\tilde{\vct{y}} = \Exp\cA\vct{x}$ for suitable $\vct{x} \in \set{K}$.
    Since $p$ is a \polation{} of $f$, we find $p(\vct{y}) = y_{\vct{0}}^d f(\vct{x})$, which is positive by assumption on $f,\set{K}$.
    We therefore have that $p$ is positive on $\set{T}\setminus\{\vct{0}\}$.
    
    By Theorem \ref{thm:dpReduction} there exists a $Q' \subset Q$ that is finite and where $p$ is positive on $\set{T}' \setminus \{\vct{0}\}$ for $\set{T}' := \{ \vct{y} \,:\, q(\vct{y}) \geq 0 \, \, \forall \, \, q \in Q' \}$.
    We are free to assume $Q' = Q(\X)' \cup Q(G)$ where $Q(\X)' \subset Q(\X)$ includes the constant polynomial $\vct{y} \mapsto 1$.
    By Theorem \ref{thm:complete:dpFinite}, there exists an $r \in \N$ and homogeneous polynomials $\{h_q\}_{q\in Q'}$ on $\R^{\cA}$ with nonnegative coefficients where
    \[
        (\textstyle{\sum_{\evec \in \cA} y_{\evec})^r} p(\vct{y}) = \sum_{q \in Q'} h_q(\vct{y})q(\vct{y}).
    \]
    From property (i) of $Q(\X)$ we know that each constraint polynomial $q \in Q(\X)'$ has at most two terms.
    In addition, property (ii) of $Q(\X)$ tells us that the \sigation{} of any $q \in Q(\X)'$ is $\X$-nonnegative.
    It is easily verified that all $\X$-nonnegative signomials with at most two terms are $\X$-AGE.
    We may therefore apply Lemma \ref{lem:dp2sage} to obtain
    \begin{equation*}
        \textstyle{(\sum_{\evec\in\cA}\e^{\evec})^r}f = \lambda_f + \sum_{q\in Q(G)} \lambda_q g_q
    \end{equation*}
    for signomials $g_q(\vct{x}) = q(\Exp \cA \vct{x})$, posynomials $\lambda_q(\vct{x}) = \mu_q(\Exp \cA \vct{x})$, and an $\X$-SAGE $\lambda_f \in \R[\cA]$.
    We complete the proof by noting that $\{ g_q \}_{q\in Q(G)} = G$.
\end{proof}

We emphasize that the decomposition promised in Theorem \ref{thm:newPositiv} makes no reference to the set $Q(\X)$ used in our proof of the theorem.
This reflects how such a decomposition exists for given $r \in \N$ if (but not only if) there are \textit{any} polynomials $Q(\X)$ satisfying Lemma \ref{lem:dp2sage} where the \polation{} of $f$ admits a Dickinson-Povh certificate over $\{ \vct{y} \,:\, q(\vct{y}) \geq 0 \text{ for all } q \in Q(G) \cup Q(\X) \}$ with exponent $r$.
So by virtue of using SAGE certificates we do not need to construct $Q(\X)$ explicitly, and in fact we automatically do at least as well as choosing the \textit{best possible} $Q(\X)$ consistent with Lemmas \ref{lem:polynomialize_X} and Theorems \ref{thm:complete:dpFinite} and \ref{thm:dpReduction}.

\subsection{Proof of Lemma \ref{lem:dp2sage}}\label{subsec:proof_lem_dp2sage}

Let $r \in \N$ be such that the stated polynomials $\mu_q$ exist, and let $\lambda_q, g_q$ be the \sigation{}s given in the lemma statement.
Since $\mu_q$ are polynomials with nonnegative coefficients, the \sigation{}s $\lambda_q$ are posynomials.
Set $\lambda_f = \sum_{q \in Q_2} \lambda_q g_q$.
We are given that the \sigation{} $g_q$ of any $q \in Q_2$ is $\X$-AGE.
Since the product of an $\X$-AGE function with a posynomial is $\X$-SAGE, and sums of such products are likewise $\X$-SAGE, we find that the stated $\lambda_f$ is $\X$-SAGE.
Completing the proof is a matter of purely algebraic identifications.
Namely,
\begin{align}
    \textstyle{(\sum_{\evec\in\cA} \e^{\evec}(\vct{x}))^r}f(\vct{x}) 
        &= \textstyle{(\sum_{\evec\in\cA} \e^{\evec}(\vct{x}))^r} p(\Exp \cA \vct{x})  \nonumber \\
        &= \sum_{q\in Q} \mu_q(\Exp \cA \vct{x}) q(\Exp \cA \vct{x}) = \lambda_f(\vct{x}) + \sum_{q\in Q_1} \lambda_g(\vct{x}) g_q(\vct{x}), \nonumber
\end{align}
where the last equality decomposed the sum over $Q = Q_1 \cup Q_2$ and applied the definitions of $\lambda_f$, $\lambda_q$, and $g_q$.

\subsection{Proof of Lemma \ref{lem:polynomialize_X}}\label{subsec:proof_lem_polynomialize_X}

We build up this set of polynomials incrementally.
To avoid clutter we use the symbol $Q$ rather than $Q(\X)$ for the proof, and we set $\set{T} := \{ \vct{y} \,:\, q(\vct{y}) \geq 0 \, \forall \, q\in Q\}$ for the current value of $Q$.
Begin by initializing $Q = \{ \vct{y} \mapsto y_{\evec} \,:\, \evec \in \cA \}$, so that $\set{T} = \R^{\cA}_+$.
Then update $Q$ to include the $2(|\cA|-1)$ linear functions
\[
q_{\evec}^{+}(\vct{y}) = y_{\vct{0}}\left(\max_{\vct{x}\in \X}\e^{\evec}(\vct{x})\right) - y_{\evec} \quad \text{and } \quad  q_{\evec}^{-}(\vct{y}) = y_{\evec} - y_{\vct{0}}\left(\min_{\vct{x}\in \X} \e^{\evec}(\vct{x})\right) 
\]
for $\evec \in \cA \setminus \{\vct{0}\}$.
The functions $q_{\evec}^{+}$ ensure that any vector $\vct{y} \in \set{T}$ with $y_{\vct{0}} = 0$ necessarily satisfies $\vct{y} = \vct{0}$.
Conversely, the functions $q_{\evec}^{-}$ ensure that when $y_{\vct{0}} > 0$ we have $\vct{y} > \vct{0}$.
When considered together, we have that if $\vct{y}$ is a nonzero vector in $\set{T}$, then $\vct{y} > \vct{0}$, and so the set of polynomials $Q$ already satisfies property (iii) in the lemma statement.

We turn to property (ii).
Recall the notation where $\vct{\delta}_{\evec}$ is the standard basis vector in $\R^{\cA}$ corresponding to $\evec \in \cA$.
Then $\vct{y}$ belongs to $\Exp\cA\X$ if an only if there exists a $\vct{z} \in \cA\X$ where $y_{\evec} = \e^{\vct{\delta}_{\evec}}(\vct{z})$.
The set $\cA\X$ is compact and convex, therefore by a continuity argument (or a direct application of \cite[Theorem 3.1]{dcST2020}) there exists a set $S \subset \Z^{\cA} \times \R$ where $\vct{z} \in \cA\X$ holds if and only if $\langle\vct{a},\vct{z}\rangle \leq \log b$ for all $(\vct{a}, \log b) \in S$.
We can take $S$ to be countable by always choosing $\log b = \max\{ \langle \vct{a}, \cA\vct{x}\rangle \,:\, \vct{x}\in \X\}$.
For given $(\vct{a},\log b) \in S$, take componentwise maximums $\vct{\gamma} = \vct{0} \wedge \vct{a}$ and $\vct{\kappa} = \vct{0} \wedge (-\vct{a})$, so the inequality $\langle\vct{a},\vct{z}\rangle \leq \log b$ is equivalent to $\e^{\vct{\gamma}}(\vct{z}) \leq b \e^{\vct{\kappa}}(\vct{z})$.
For each such signomial inequality there is a polynomial inequality $\vct{y}^{\vct{\gamma}} \leq b \vct{y}^{\vct{\kappa}}$ that is equivalent in the relevant regime $\vct{y} \in \set{T}, y_{\vct{0}} = 1$.
Setting  $v = \|\vct{a}\|_1 - \|\vct{\kappa}\|_1$ and $u = \|\vct{a}\|_{1} - \|\vct{\gamma}\|_1$, we homogenize the polynomial inequality defined above to an equivalent form $q(\vct{y}) =  b y_{\vct{0}}^{v} \vct{y}^{\vct{\kappa}} -  y_{\vct{0}}^{u} \vct{y}^{\vct{\gamma}} \geq 0$.
We finalize $Q$ by updating it to contain all homogeneous polynomials obtained in this way.
Since all reformulations employed here were reversible over $\R^{\cA}_{++}$, we have (ii): $\Exp \cA\X = \{ \vct{y} : q(\vct{y}) \geq 0 \, \forall \, q \in Q, y_{\vct{0}} = 1\}$.

As property (i) holds by construction, the proof is complete.

\begin{remark}\label{rem:discuss_QX_ideal}
    The exponents of the polynomials in $Q$ were derived from halfspaces that contain the compact convex set $\cA\X$.
    Since $\cA\X$ is low-dimensional in general, some of these halfspaces can come together to form \textit{hyperplanes} containing $\cA\X$.
    Let $H \subset \N^{\cA}$ denote the set of all integral normal vectors of hyperplanes that contain $\cA\X$.
    For $\vct{a} \in H$ we can use the construction described above to obtain a polynomial $\vct{y} \mapsto q(\vct{y}) = b y_{\vct{0}}^{v}\vct{y}^{\vct{\kappa}} - y_{\vct{0}}^{u}\vct{y}^{\vct{\gamma}}$ where $q(\exp \cA \vct{x}) = 0$ for all $\vct{x}$ in $\X$.
    If $\X$ is full-dimensional then the real locus of these polynomials is more or less the smallest toric variety that contains $\set{Y} := \exp\cA \R^n$.
    It is possible that $\set{Y}$ is poorly approximated by a variety; in this case we are leveraging compactness to provide a \textit{local} description for $\exp\cA\X$ in terms of infinitely many polynomial inequalities.
    Prior assumptions from \cite{CS2016,WJYP2020} that $\cA \subset \Q^n$ were used to construct polynomial equations for describing $\exp\cA\R^n$ as the intersection of a variety with the positive orthant.
\end{remark}

\section{A complete hierarchy of lower bounds}\label{sec:lower_bounds}

This section demonstrates how the concept of signomial rings leads to improved methods for lower-bounding and solving nonconvex optimization problems.
Formally, given a finite set of signomials $\{f \} \cup G$ and a closed convex set $\X$, we would like to solve
\begin{equation}\label{prob:sig_optim_prob}
    f_{\set{K}}^\star = \inf_{\vct{x} \in \set{K}} f(\vct{x}) \quad \text{where} \quad \set{K} = \{ \vct{x}\in \X \,:\, g(\vct{x}) \geq 0 \text{ for all } g \text{ in } G \}.
\end{equation}
Our high-level approach here is quite standard.
We want certificates that shifted signomials $f - \gamma$ are nonnegative on $\set{K}$, and such certificates are available to us through Theorem \ref{thm:newPositiv}.
In order to implement this idea we just need to grade the certificates according to largest $\cA$-degree of the constituent signomials.

We recall two essential definitions from Section \ref{sec:sig_rings}.
First, the $\cA$-degree of a signomial $h$ is the smallest integer $d$ for which $\supp(h) \subset \cA_{d}$.
Second, for a given signomial $h$ and positive integer $d$, the set $A := \invsupp_{d}(h)$ is the largest $A\subset\cA_d$ for which $\deg_{\cA}(h \e^{\bevec} ) \leq d$ for every $\bevec \in A$.

\begin{definition}\label{def:lowerHierarchy}
    Given an integer $d$ where $r := d - \deg_{\cA}(f) \geq 0$, the $\cA$-degree $d$ SAGE bound for Problem~\eqref{prob:sig_optim_prob} is
    \begin{align}
        \struc{f_{\set{K}}^{(d)}} := \sup ~\gamma~\text{s.t.}
            ~&~ ({\textstyle\sum_{\evec\in\cA} \e^{\evec}})^{r}(f - \gamma) - {\textstyle\sum_{g \in G} \lambda_g g} \text{ is } \X\text{-SAGE} \label{eq:hierarchy},\\
            ~&~ \gamma \in \R, \text{ and } \lambda_g \in {\csage_{\X}}\left(\invsupp_{d}(g)\right) \text{ for each } g \in G.  \nonumber
    \end{align}
    When $d < \deg_{\cA}(f)$, we set $f_{\set{K}}^{(d)} = -\infty$.
\end{definition}

When the hierarchy is applied to problems with an equality constraint $g(\vct{x}) = 0$, one simply uses an unconstrained multiplier $\lambda_g \in \spann\{\e^{\evec}\,:\, \evec \in \invsupp_{d}(g)\}$.
The canonical ``lowest level'' of the hierarchy is to take $d = \deg_{\cA}(f)$.
Note that we make no assumptions about the $\cA$-degree of constraint signomials, because it is possible that $\deg_{\cA}(g) > d$ and yet $\invsupp_{d}(g)$ is nonempty (see Example \ref{ex:multiplier_subspace}).
Here is an important point we reference later in this section.

\begin{remark}\label{rem:lowest_level}
    When $G = \emptyset$, the bound provided by the lowest level of the hierarchy \eqref{eq:hierarchy} is independent of the underlying signomial ring.
    That is, if we set $d = \deg_{\cA}(f)$ and compute $f_{\set{K}}^{(d)}$ when working in $\R[\cA]$, then the numeric value of $f_{\set{K}}^{(d)}$ remains unchanged if we replace $\cA \leftarrow \cA'$ (provided $f$ belongs to both $\R[\cA]$ and $\R[\cA']$).
\end{remark}

And here is our main theoretical result for the hierarchy.

\begin{corollary}\label{cor:lowerComplete}
    The sequence $f_{\set{K}}^{(1)},f_{\set{K}}^{(2)},\ldots $ is nondecreasing and bounded above by $f_{\set{K}}^\star$.
    If the signomials $\{f\} \cup G$ belong to $\R[\cA]$ and $\X$ is compact, then
    \[
        \lim_{d \to\infty } f_{\set{K}}^{(d)} = f_{\set{K}}^\star.
    \]
\end{corollary}
\begin{proof}
    The sequence is nondecreasing because $\invsupp_{d}(g) \subset \invsupp_{d+1}(g)$ and $\Csage_{\X}(A) \subset \Csage_{\X}(A')$ whenever $A \subset A' \subset \R^n$.
     That is, the feasible sets grow with $d$.
     The sequence is bounded above by $f_{\set{K}}^\star$ because every feasible solution certifies $f(\vct{x}) \geq \gamma$ for all $\vct{x} \in \set{K}$.
     Under the assumptions on $\X$ and $\R[\cA]$, convergence to $f_{\set{K}}^\star$ follows from Theorem \ref{thm:newPositiv} and the fact that posynomials are trivially $\X$-SAGE.
\end{proof}

Corollary \ref{cor:lowerComplete} is the first completeness result for minimizing an arbitrary signomial subject to constraints given by a compact convex set and a conjunction of arbitrary (but finitely many) signomial inequalities.
It is also the first completeness result for a hierarchy that uses conditional SAGE in the presence of nonconvex constraints.
As a practical matter, we recall a question posed in \cite[\S 5.4]{MCW2019}: when designing SAGE-based hierarchies for signomial optimization, how should we decide the support of a generalized Lagrange multiplier $\lambda_g$ by considering properties of the constraint signomial $g$?
Our Corollary \ref{cor:lowerComplete} says this can be done in any way that respects the structure of a signomial ring.

The rest of this section explores the practicality of our hierarchy through three examples, primarily stated in terms of variables $\vct{t} = \exp \vct{x}$.
The first problem illustrates various considerations in choosing a signomial ring.
The second problem is adapted from stability analysis of a chemical reaction network described in \cite{Pantea2012} and \cite{MurrayPhD}; applying our methods to this problem result in more than a 500x speedup over the Lasserre hierarchy.
Our last problem is a benchmark signomial program which encodes the design of a chemical reactor \cite{BW1969-sig-ChemE}.
We use Section \ref{subsec:probsca} to prove an invariance property of our hierarchy before providing the reactor design problem in Section \ref{subsec:CSTR}.
Our methods outperform SCIP, BARON, and ANTIGONE on this problem. 

We implemented the hierarchy \eqref{eq:hierarchy} using the \texttt{sageopt} python package \cite{sageopt} and we used GloptiPoly3 to access the Lasserre hierarchy \cite{gloptipoly}.
MOSEK 9.2 was the underlying numerical solver for all such convex relaxations; see \cite{DA2021} for the mathematical statement of MOSEK's algorithm.
All computation was performed on a machine with a Core i7-1065G7 CPU (4 cores at 1.30GHz) and 16 GB DDR4 RAM (3733 MT/s).

\subsection{Low degree polynomial optimization}

If all signomials in \eqref{prob:sig_optim_prob} have integer exponents and finite lower bounds on the decision variable $\vct{x}$, then the problem can be written with polynomials in $\vct{t}$ by clearing denominators.
The following nonconvex quadratic program was obtained by applying this procedure to \cite[Problem 23]{RM1978}.

\begin{align}
    \min_{\vct{t} \in \R^5_{++}} 
                & ~~~ 5.3578\, t_3^2 + 0.8357\, t_1^{} t_5^{} + 37.2393 t_1^{} 
                \label{prob:ex1} \\
    \text{s.t }
                & \begin{rcases}
                    &  0.06663\, t_2^{} t_5^{} -0.02584\, t_3^{} t_5^{} + 0.0734\, t_1^{} t_4^{} + 1000 \geq 0 \\
                    &  0.33085\, t_3^{} t_5^{} - 0.853007\, t_2^{} t_5^{} - 0.09395\, t_1^{} t_4^{} + 1000 \geq 0 \\
                    &  0.4200\, t_1^{} t_2^{} + 0.30586\, t_3^2 + t_2^{} t_5^{} - 1330.3294 \geq 0 \\
                    &   0.2668\, t_1^{} t_3^{} + 0.40584\, t_3^{} t_4^{} + t_3^{} t_5^{} - 2275.1326 \geq 0
                \end{rcases} G_{\text{nonconvex}}
                \nonumber \\
                & ~~~\, 1000 - 0.24186\, t_2^{} t_5^{} - 0.10159\, t_1^{} t_2^{} - 0.07379\, t_3^2 \geq 0
                \nonumber \\
                & ~~~\, 1000 - 0.29955\, t_3^{} t_5^{} - 0.07992\, t_1^{} t_3^{} - 0.12157\, t_3^{} t_4^{}  \geq 0 
                \nonumber \\
                & \begin{rcases}
                    & (102, 45, 45, 45, 45) - \vct{t} \geq \vct{0}\\
                    & \vct{t} - (78, 33, 27, 27, 27) \geq \vct{0}
                \end{rcases} G_{\text{box}} \nonumber
\end{align}

We have labeled the set of inequality constraints that are nonconvex in $\vct{x}$ as $G_{\text{nonconvex}}$ and use $G_{\text{box}}$ for the signomials that imply box constraints on $\vct{x}$.
We use $G_{\text{all}}$ to refer to all constraints appearing in  \eqref{prob:ex1}.
Applying solution recovery to the SAGE relaxations discussed below shows that the optimal solution to this problem is
    $\vct{t}^\star \approx (78,~ 33,~ 29.99574,~ 45,~ 36.77533)$
with optimal objective $f_{\set{K}}^\star \approx 10122.4932$

Problem \eqref{prob:ex1} lets us illustrate the effect of considering different signomial rings and different sets of ``algebraic'' constraints $G$.
While exploring these effects, we fix
\[
    \X = \{\vct{x}\,:\, g(\vct{x}) \geq 0\, \text{ for all }\, g \in G_{\text{all}} \setminus G_{\text{nonconvex}} \}.
\]
We examine three cases where we set $G$ to $G_{\text{all}}$, to $G_{\text{all}} \setminus G_{\text{box}}$, and to $G_{\text{nonconvex}}$.
For each choice of $G$ we consider two types of signomial rings.
For the \textit{naive rings} we take $\cA$ as the smallest set so every signomial in $\{ f \} \cup G$ has $\cA$-degree one.
The naive rings have generating sets of size $19$, $15$, and $12$ (as $G$ gets smaller).
We also consider a \textit{natural ring} $\cA = \{\vct{0},\vct{\delta}_1,\ldots,\vct{\delta}_n\}$ that reflects how \eqref{prob:ex1} is polynomial in $\vct{t}$.

Performance data for the SAGE relaxations is given in Tables  \ref{tab:ex1:natural} and \ref{tab:ex1:naive}. 
We see finite convergence for the hierarchy in four out of the six choices of $(G,\cA)$.
The best bound at each hierarchy level used $G_{\text{all}}$.
This reflects a known phenomenon where incorporating a constraint in an explicit algebraic way can improve bounds even when the constraint is nominally accounted for in the set $\X$.
Another key point from the data is that the solver runtimes scale more gracefully when using the natural ring compared to the naive ring.
This is to be expected, since the natural ring is smaller than the naive rings.

\begin{table}[ht!]
    \caption{Natural-ring SAGE bounds and solver runtimes for problem \eqref{prob:ex1}.}
    \small\centering
    \begin{tabular}{c|ccc|ccc|} 
    & \multicolumn{3}{c}{$\cA$-degree $d$ SAGE bounds} & \multicolumn{3}{|c|}{solver runtimes (s)} \\
    $d$ & $G_{\text{all}}$ & $G_{\text{all}} \setminus G_{\text{box}}$ & $G_{\text{nonconvex}}$ & $G_{\text{all}}$ & $G_{\text{all}} \setminus G_{\text{box}}$ & $G_{\text{nonconvex}}$ \\ \hline
    2 & 10022.940 & \textcolor{white}{1}9322.848 & \textcolor{white}{1}9322.849 & 0.070 & 0.015 & 0.015 \\
    3 & {10122.493} & \textcolor{white}{1}9964.326 & \textcolor{white}{1}9954.832 & 0.588 & 0.184 & 0.132 \\
    4 & - & {10122.493} & 10074.250 & - & 1.314 & 1.368 \\ 
    \end{tabular}
    \label{tab:ex1:natural}
\end{table}

\begin{table}[ht!]
    \caption{Naive-ring SAGE bounds and solver runtimes for problem \eqref{prob:ex1}.}
    \small\centering
    \begin{tabular}{c|ccc|ccc|} 
    & \multicolumn{3}{c}{$\cA$-degree $d$ SAGE bounds} & \multicolumn{3}{|c|}{solver runtimes (s)} \\
    $d$ & $G_{\text{all}}$ & $G_{\text{all}}\setminus G_{\text{box}}$ & $G_{\text{nonconvex}}$ & $G_{\text{all}}$ & $G_{\text{all}}\setminus G_{\text{box}}$ & $G_{\text{nonconvex}}$ \\ \hline
    1 & 10022.929 &  \textcolor{white}{1}9322.848 &  \textcolor{white}{1}9322.849 & 0.045 & 0.015 & 0.015 \\
    2 & { 10122.493} & 10069.946 & 10059.838 & 1.600 & 0.289 & 0.338 \\
    3 & - & {10122.493} & 10112.300 & - & 7.939 & 7.918 \\ 
    \end{tabular}
    \label{tab:ex1:naive}
\end{table}

We consider one more SAGE relaxation before finishing this example.
Using $\X=\R^n$, $G = G_{\text{all}}$, the natural ring, and $d=3$, we solve \eqref{prob:ex1} exactly with a relative entropy program that is solved in only 0.27 seconds.
This shows how hierarchy \eqref{eq:hierarchy} has practical value even for noncompact sets $\X$.
Moreover, the 0.27 second solve time with this SAGE approach is comparable to the 0.23 second solve time required by using the first level of the Lasserre hierarchy (which also solves \eqref{prob:ex1} exactly).
The next section considers a high-degree polynomial optimization problem for which our methods are orders of magnitude more efficient than the Lasserre hierarchy.

\subsection{Minimizing a polynomial from chemical reaction network theory}

A \textit{chemical reaction network} (CRN) is a continuous-time dynamical system where the state variables represent the concentrations of various \textit{chemical species} in a shared environment \cite{Yu2018}.
CRNs are typically defined by assuming the law of mass-action kinetics.
Here, the rate at which a reaction occurs is proportional to the concentrations of the reactants.
Under mass-action kinetics, a CRN is specified by a polynomial map $p : \R^n_{++} \times \R^{m}_{++} \to \R^n$
\begin{equation}\label{eq:def_crn}
    \frac{\diff}{\diff t}\vct{s}(t) = p(\vct{s}(t);\vct{r})
\end{equation}
-- where $\vct{r}$ specifies the proportionality constants for the rate of each reaction.

Significant effort has been devoted to understanding when a given CRN exhibits \textit{multistationarity} (i.e., has multiple fixed points) over a set $\set{S} \subset \R^n_{++}$.
We refer the reader to \cite[\S 1]{Conradi2019} for references on this line of work.
Variations of this problem ask if a system as \textit{capacity for multistationarity}: can we add a constant vector-valued offset to $p$ so \eqref{eq:def_crn} has multiple fixed points in $\set{S}$?
Important work by Pantea, Koeppl, and Craciun \cite{Pantea2012} shows that for most systems encountered in applications, one can decide if a system has capacity for multistationarity for some $\vct{r} \in \set{R}$ by checking if $q(\vct{r},\vct{s}) \coloneqq \det\operatorname{Jac} p(\vct{s};\vct{r})$ never vanishes on $\set{R}\times\set{S}$. (Where the Jacobian is only taken with respect to $\vct{s}$.)

Here we consider the stylized problem of minimizing a signomial obtained from such a test for capacity for multistationarity.
The underlying CRN is from \cite{Pantea2012}.
It consists of six species and twenty reaction rates, although only a handful of these twenty six variables affect capacity for multistationarity over $\set{S} = \R^6_{++}$.
Following the reductions described in \cite{Pantea2012}, we are left with a polynomial $\hat{q}$ in only nine variables.
The data for this polynomial's monomial exponents and coefficients is given in Table~\ref{tab:crn_poly}.
These exponents and coefficients remain the same when we pass to the signomial parameterization $f(\vct{x}) \coloneqq \hat{q}(\exp\vct{x})$.

\begin{table}
\footnotesize\centering
\caption{Monomial exponents and coefficients for the polynomial $\hat{q}$ obtained from our example in chemical reaction network theory. The first seven variables correspond to certain reaction rate parameters and the last two variables correspond to species concentrations.}
\begin{tabular}{|ccccccccc|c||ccccccccc|c|}
$\cdot$ & 1 & $\cdot$ & 1 & 1 & $\cdot$ & 1 & $\cdot$ & $\cdot$ &   1 & 1 & $\cdot$ & $\cdot$ & 1 & 1 & $\cdot$ & 1 & 1 & $\cdot$ &   1 \\
$\cdot$ & $\cdot$ & $\cdot$ & 1 & 1 & $\cdot$ & 1 & $\cdot$ & $\cdot$ &   1 & 1 & $\cdot$ & $\cdot$ & 1 & 1 & $\cdot$ & $\cdot$ & 1 & $\cdot$ &   1 \\
$\cdot$ & 1 & $\cdot$ & $\cdot$ & 1 & $\cdot$ & 1 & $\cdot$ & $\cdot$ &   1 & $\cdot$ & 1 & 1 & $\cdot$ & $\cdot$ & $\cdot$ & $\cdot$ & $\cdot$ & 1 &   1 \\
$\cdot$ & $\cdot$ & $\cdot$ & $\cdot$ & 1 & $\cdot$ & 1 & $\cdot$ & $\cdot$ &   1 & $\cdot$ & $\cdot$ & 1 & $\cdot$ & $\cdot$ & $\cdot$ & $\cdot$ & $\cdot$ & 1 &   1 \\
$\cdot$ & 1 & $\cdot$ & 1 & $\cdot$ & $\cdot$ & 1 & $\cdot$ & $\cdot$ &   1 & $\cdot$ & 1 & $\cdot$ & $\cdot$ & $\cdot$ & 1 & $\cdot$ & $\cdot$ & 1 &   4 \\
$\cdot$ & $\cdot$ & $\cdot$ & 1 & $\cdot$ & $\cdot$ & 1 & $\cdot$ & $\cdot$ &   1 & $\cdot$ & $\cdot$ & $\cdot$ & $\cdot$ & $\cdot$ & 1 & $\cdot$ & $\cdot$ & 1 &   4 \\
$\cdot$ & 1 & $\cdot$ & $\cdot$ & $\cdot$ & $\cdot$ & 1 & $\cdot$ & $\cdot$ &   1 & $\cdot$ & 1 & 1 & $\cdot$ & 1 & $\cdot$ & $\cdot$ & $\cdot$ & 1 &   1 \\
$\cdot$ & $\cdot$ & $\cdot$ & $\cdot$ & $\cdot$ & $\cdot$ & 1 & $\cdot$ & $\cdot$ &   1 & $\cdot$ & $\cdot$ & 1 & $\cdot$ & 1 & $\cdot$ & $\cdot$ & $\cdot$ & 1 &   1 \\
$\cdot$ & 1 & $\cdot$ & 1 & 1 & $\cdot$ & $\cdot$ & $\cdot$ & $\cdot$ &   1 & $\cdot$ & 1 & 1 & $\cdot$ & $\cdot$ & $\cdot$ & 1 & $\cdot$ & 1 &   1 \\
$\cdot$ & $\cdot$ & $\cdot$ & 1 & 1 & $\cdot$ & $\cdot$ & $\cdot$ & $\cdot$ &   1 & $\cdot$ & $\cdot$ & 1 & $\cdot$ & $\cdot$ & $\cdot$ & 1 & $\cdot$ & 1 &   1 \\
$\cdot$ & 1 & $\cdot$ & $\cdot$ & 1 & $\cdot$ & $\cdot$ & $\cdot$ & $\cdot$ &   1 & $\cdot$ & 1 & $\cdot$ & 1 & $\cdot$ & 1 & $\cdot$ & $\cdot$ & 1 &   4 \\
$\cdot$ & $\cdot$ & $\cdot$ & $\cdot$ & 1 & $\cdot$ & $\cdot$ & $\cdot$ & $\cdot$ &   1 & $\cdot$ & $\cdot$ & $\cdot$ & 1 & $\cdot$ & 1 & $\cdot$ & $\cdot$ & 1 &   4 \\
$\cdot$ & 1 & $\cdot$ & 1 & $\cdot$ & $\cdot$ & $\cdot$ & $\cdot$ & $\cdot$ &   1 & $\cdot$ & 1 & 1 & $\cdot$ & 1 & $\cdot$ & 1 & $\cdot$ & 1 &   1 \\
$\cdot$ & $\cdot$ & $\cdot$ & 1 & $\cdot$ & $\cdot$ & $\cdot$ & $\cdot$ & $\cdot$ &   1 & $\cdot$ & $\cdot$ & 1 & $\cdot$ & 1 & $\cdot$ & 1 & $\cdot$ & 1 &   1 \\
$\cdot$ & 1 & $\cdot$ & $\cdot$ & $\cdot$ & $\cdot$ & $\cdot$ & $\cdot$ & $\cdot$ &   1 & 1 & $\cdot$ & 1 & $\cdot$ & $\cdot$ & $\cdot$ & 1 & 2 & $\cdot$ &   1 \\
$\cdot$ & $\cdot$ & $\cdot$ & $\cdot$ & $\cdot$ & $\cdot$ & $\cdot$ & $\cdot$ & $\cdot$ &   1 & 1 & $\cdot$ & 1 & $\cdot$ & $\cdot$ & $\cdot$ & $\cdot$ & 2 & $\cdot$ &   1 \\
1 & $\cdot$ & $\cdot$ & $\cdot$ & $\cdot$ & $\cdot$ & 1 & 1 & $\cdot$ &   1 & 1 & $\cdot$ & 1 & $\cdot$ & $\cdot$ & $\cdot$ & $\cdot$ & 1 & 1 &   1 \\
1 & $\cdot$ & $\cdot$ & $\cdot$ & $\cdot$ & $\cdot$ & $\cdot$ & 1 & $\cdot$ &   1 & 1 & $\cdot$ & $\cdot$ & $\cdot$ & $\cdot$ & 1 & $\cdot$ & 1 & 1 &   4 \\
$\cdot$ & $\cdot$ & 1 & $\cdot$ & $\cdot$ & $\cdot$ & 1 & 1 & $\cdot$ &   1 & 1 & $\cdot$ & 1 & $\cdot$ & $\cdot$ & $\cdot$ & 1 & 1 & 1 &   1 \\
$\cdot$ & $\cdot$ & 1 & $\cdot$ & $\cdot$ & $\cdot$ & $\cdot$ & 1 & $\cdot$ &   1 & 1 & $\cdot$ & $\cdot$ & 1 & $\cdot$ & 1 & $\cdot$ & 1 & 1 &   4 \\
1 & $\cdot$ & $\cdot$ & 1 & $\cdot$ & $\cdot$ & 1 & 1 & $\cdot$ &   1 & 1 & $\cdot$ & 1 & $\cdot$ & 1 & $\cdot$ & $\cdot$ & 1 & 1 &   -1 \\
1 & $\cdot$ & $\cdot$ & 1 & $\cdot$ & $\cdot$ & $\cdot$ & 1 & $\cdot$ &   1 & 1 & $\cdot$ & 1 & $\cdot$ & 1 & $\cdot$ & 1 & 1 & 1 &   -1 \\
1 & $\cdot$ & $\cdot$ & $\cdot$ & 1 & $\cdot$ & 1 & 1 & $\cdot$ &   1 & $\cdot$ & 1 & 1 & $\cdot$ & $\cdot$ & 1 & $\cdot$ & $\cdot$ & 2 &   4 \\
1 & $\cdot$ & $\cdot$ & $\cdot$ & 1 & $\cdot$ & $\cdot$ & 1 & $\cdot$ &   1 & $\cdot$ & $\cdot$ & 1 & $\cdot$ & $\cdot$ & 1 & $\cdot$ & $\cdot$ & 2 &   4 \\
$\cdot$ & 1 & 1 & $\cdot$ & $\cdot$ & $\cdot$ & 1 & 1 & $\cdot$ &   1 & 1 & $\cdot$ & 1 & $\cdot$ & $\cdot$ & 1 & $\cdot$ & 1 & 2 & 4 \\
$\cdot$ & 1 & 1 & $\cdot$ & $\cdot$ & $\cdot$ & $\cdot$ & 1 & $\cdot$ & 1 & & & & & & & & & & \\
\end{tabular}
\label{tab:crn_poly}
\end{table}

Our choice of domain in this minimization example is motivated by \cite[\S 4.3]{MurrayPhD}, which details the use of conditional SAGE to decide positivity of $f$ over sets
\begin{equation}\label{eq:crn_opt_feas_set}
    \X = \{ \vct{x} \in \R^9 \,:\, x_5 = \log a,~ -\log b \leq x_i \leq \log b ~\forall~ i \in [7]\setminus\{5\} \}  
\end{equation}
for $(a,b)$ in a 50-by-50 linearly spaced grid of a box $[1, 10] \times [10^{1/4}, 10^{3/4}]$.
In particular, we apply hierarchy \eqref{eq:hierarchy} with $G = \emptyset$ to minimize $f$ over $\X$ when $(a,b) = (7.06, 2.41)$.
Using this domain maximizes the gap between the lowest-level SAGE bound and $f_{\X}^{\star}$ among all $\X$ of the form \eqref{eq:crn_opt_feas_set}, for the indicated parameter space.
Recall from Remark~\ref{rem:lowest_level} that the lowest-level SAGE bound is independent of $\R[\cA]$ for problems with $G = \emptyset$.

We approach this problem with the natural ring generated by $\cA_{\mathrm{nat}} = \{\vct{0}\}\cup \{\vct{\delta}_i\}_{i=1}^n$, the naive ring generated by $\cA_{\text{naive}} = \supp(f)$, and an intermediate ring generated by ``$\cA_{\mathrm{int}}$,'' which we form by updating $\cA_{\mathrm{nat}}$ to include the two $\evec\in\supp(f)$ with negative coefficients (see Table~\ref{tab:crn_poly}).
The results are reported in Table \ref{tab:exChem}.
The standard solution recovery procedure \cite[Algorithm 1]{MCW2019} certifies finite convergence to $f_{\X}^\star \approx 22.8321$ for the intermediate and naive rings.
It is of note that while the naive-ring relaxations with $G = \emptyset$ were already defined in \cite{MCW2019} and analyzed in \cite{WJYP2020}, using the intermediate ring let us compute $f_{\X}^\star$ in one third of the time than the naive ring (1.11 and 3.46 seconds respectively).
It would be a material advance in signomial and polynomial optimization if one could develop a theory on how to best choose the ring behind a SAGE relaxation.

\begin{table}[ht!]
    \caption{SAGE bounds and solver runtimes for minimization of the polynomial obtained from a chemical reaction network. We indicate the level of the hierarchy \eqref{eq:hierarchy} by specifying $r = d - \deg_{\cA}(f)$, rather than specifying $d$. }
    \small\centering
    \begin{tabular}{c|ccc|ccc|} 
    & \multicolumn{3}{c}{SAGE bounds} & \multicolumn{3}{|c|}{solver runtimes (s)} \\
    $r$ & $\cA_{\mathrm{nat}}$ & $\cA_{\mathrm{int}}$ & $\cA_{\text{naive}}$ & $\cA_{\mathrm{nat}}$ & $\cA_{\mathrm{int}}$ & $\cA_{\text{naive}}$ \\ \hline
    0 & 18.1596  & 18.1596  & 18.1596  & 0.0344 & 0.0301 & 0.0321 \\
    1 & 18.7188  & 22.8321 &  22.8321 & 1.0541 & 1.1123 & 3.4648 \\
    2 & 19.7375  &  - & -  & 49.2000 & - & - \\ 
    \end{tabular}
    \label{tab:exChem}
\end{table}

When we specified this problem in polynomial data to GloptiPoly3, the lowest-level Lasserre relaxation was degree six and returned a bound of $-\infty$ after 14.14 seconds.
The degree eight Lasserre relaxation ran for 572 seconds, reported a marginally infeasible bound of 22.8350, and failed to extract a solution.
This behavior is not unique to our choice of $(a,b)$ in \eqref{eq:crn_opt_feas_set}; the multilinear structure in $\hat{q}$ means that a degree six Lasserre relaxation will return a bound of $-\infty$ for minimizing $\hat{q}$ over any set specified by linear inequalities.
Moving to a degree eight Lasserre relaxation for this nine-variable problem results in a very large semidefinite program.
It is of interest to determine if the recently developed TSSOS (or Chordal TSSOS) approach to sums of squares can take advantage of the sparsity pattern seen in polynomials such as this \cite{WML2021-TSSOS,WML2021-chordal-TSSOS}.

\subsection{Problem scaling}\label{subsec:probsca}

When signomial programs are considered in variables $\vct{t} = \exp \vct{x}$, individual decision often differ by several orders of magnitude at optimality.
One reason for this is that signomial models typically involve physical quantities with particular choices for units.
Although it is \textit{possible} to chose units where decision variables are similarly scaled at optimality, this may not be a natural thing to do from a modeling standpoint.
This creates a need for algorithmic tools for signomial optimization that are insensitive to scaling of the variable $\vct{t}$.
The following proposition shows that our hierarchy has such scale invariance.

\begin{proposition}\label{prop:scale}
    Consider a signomial objective function $f$ and a set of constraint signomials $G \cup G'$ where $\X = \{\vct{x} \,:\, g(\vct{x}) \geq 0 \, \text{ for all }\, g \in G'\}$ is convex.
    Given a vector $\vct{b} \in \R^n$, construct translated problem data
    \begin{itemize}
        \item $f_{\vct{b}}$ defined by $f_{\vct{b}}(\vct{x}) = f(\vct{x} - \vct{b})$,
        \item $G_{\vct{b}} = \{ \vct{x} \mapsto g(\vct{x} - \vct{b}) \,:\, g \in G \}$, and
        \item $\X_{\vct{b}} = \{ \vct{x} \,:\, g(\vct{x} - \vct{b}) \geq 0 \, \text{ for all }\, g \in G'\}$.
    \end{itemize}
    Then the for every $d$ and every signomial ring $\R[\cA]$, $\struc{f_{\set{K}}^{(d)}}$ for Problem~\ref{prob:sig_optim_prob} is the same for problem data $(f,G,\X)$ and $(f_{\vct{b}},G_{\vct{b}},\X_{\vct{b}})$.    
\end{proposition}

We should emphasize that scale invariance of a SAGE bound does not mean that the behavior of algorithms for relative entropy programming are fully scale invariant.
As with any type of numerical convex optimization, changes to problem scaling in finite precision arithmetic can affect both the speed at which an REP solver converges and even whether the solver converges at all.
This proposition really shows that we are free to choose a coordinate system that works well for an REP solver without fear of changing the SAGE bound.

\begin{proof}[Proof of Proposition \ref{prop:scale}]
     Let $h$ be a signomial on $\R^n$ and consider $h_{\vct{b}}$ defined by $h_{\vct{b}}(\vct{x}) = h(\vct{x} - \vct{b})$.
     It is easy to verify that $h$ is $\X$-SAGE if and only if $h_{\vct{b}}$ is $[\X+\vct{b}]$-SAGE.
     Additionally, it is clear that $\supp(h) = \supp(h_{\vct{b}})$, and this implies both $\deg_{\cA}(h) = \deg_{\cA}(h_{\vct{b}})$ and $\invsupp_{d}(h) = \invsupp_{d}(h_{\vct{b}})$.
     Finally, observe that $\X_{\vct{b}} = \X + \vct{b}$.
     
     Using these facts we can map any feasible solution to problem \ref{eq:hierarchy} for data $(f,G,\X)$ to a feasible solution to the analogous problem for data $(f_{\vct{b}},G_{\vct{b}},\X_{\vct{b}})$ without changing $\gamma$.
     By symmetry (essentially replacing $\vct{b}$ by $-\vct{b}$) any solution to problem \ref{eq:hierarchy} for data $(f_{\vct{b}},G_{\vct{b}},\X_{\vct{b}})$ can likewise be mapped to a feasible solution for problem data $(f,G,\X)$ without changing $\gamma$.
     As the set of feasible choices for $\gamma$ is the same under these two formulations, we have that the $\cA$-degree $d$ SAGE bounds coincide.
\end{proof}

\subsection{Design of a chemical reactor system}\label{subsec:CSTR}

Here we consider the design of a chemical reactor system as described by Blau and Wilde in \cite{BW1969-sig-ChemE} and \cite{BW1971}.
This problem is a proper signomial program and we approach it through the naive ring.
None of the constraints in this problem are convex in $\vct{x}$, however we can infer convex constraints by considering the case $g(\vct{x}) \geq 0$ in each of the constraints $g(\vct{x}) = 0$.
We apply the hierarchy \eqref{eq:hierarchy} to this problem by taking $\X$ as the convex set cut out by these five inequality constraints.

\begin{align}
    \min_{\vct{t} \in \R^8_{++}} ~~& 2.0425\, t_{1}^{0.782} + 52.25\, t_{2} + 192.85\, t_{2}^{0.9} + 5.25\, t_{2}^3 + 61.465\, t_{6}^{0.467} \label{prob:ex2} \\[-0.75em]
    & \qquad + 0.01748\,  t_{3}^{1.33} / t_{4}^{0.8} + 100.7\, t_{4}^{0.546} + 3.66 {\cdot 10}^{-10} \, t_{3}^{2.85}/ t_{4}^{1.7} \nonumber \\
    & \qquad + 0.00945\, t_{5} + 1.06 {\cdot 10}^{-10}\, t_{5}^{2.8}/t_{4}^{1.8} + 116\, t_{6} - 205\, t_{6} t_{7} - 278\, t_{2}^3 t_{7} \nonumber \\[0.25em]
    \text{s.t.}~~ & 1 - 129.4/t_{2}^3 - 105/t_{6} = 0 \nonumber \\
    & 1 - 1.03\cdot 10^5 \, t_{2}^3 t_{7}/(t_{3} t_{8}) - 1.2\cdot 10^6 /(t_{3} t_{8}) = 0 \nonumber \\
    & 1 - 4.68\, t_{2}^3/t_{1} - 61.3\, t_{2}^2/t_{1} - 160.5\, t_{2}/t_{1} = 0 \nonumber \\
    & 1 - 1.79\, t_{7} - 3.02\, t_{2}^3 t_{7}/t_{6} - 35.7\, /t_{6} - 1 = 0 \nonumber \\
    & 1 - 1.22\cdot 10^{-3}\, t_{3} t_{8}/(t_{4}^{0.2} t_{5}^{0.8}) - 1.67\cdot 10^{-3}\, t_{8} t_{3}^{0.4}/ t_{4}^{0.43} \nonumber \\[-0.0em] & \qquad - 3.6 {\cdot 10}^{-5}\, t_{3} t_{8}/t_{4} - 2\cdot 10^{-3}\, t_{3} t_{8}/t_{5} - 4\cdot 10^{-3}\, t_{8} = 0 \nonumber
\end{align}

This initial problem statement is terribly scaled -- the coefficients in the objective alone span twelve orders of magnitude.
Trying to solve even lowest-level SAGE relaxation with MOSEK returns ``unknown'' status codes here.
We therefore scale the variables about the initial estimates provided in \cite{BW1971}
\[
    \vct{\tilde{t}} = (10^3,~ 10, ~ 10^5, ~ 10^2, ~10^5,~ 10^3,~ 10^{-1}, 10)
\]
and we call solvers with a scaled objective $\hat{f} \coloneqq f / 10^4$.

The coefficients in the scaled problem span only four orders of magnitude and the SAGE relaxations can be solved reliably.
We compute
\begin{align*}
    & f_{\set{K}}^{(1)} = 16377.32 \quad\text{ in }\quad ~0.13\, \text{ seconds, and} \\
    & f_{\set{K}}^{(2)} = 17462.73 \quad\text{ in }\quad 24.37 \text{ seconds.}
\end{align*}
We run solution recovery on the dual formulation for $f_{\set{K}}^{(2)}$ to obtain a point $\vct{x'}$, and refine this with the (zeroth-order) COBYLA local solver to get $\vct{x''}$ \cite{Powell1994}.
These solutions satisfy 
\begin{align*}
    f(\vct{x'}) &= 17486.52 \quad\text{ and }\quad \|G(\vct{x'})\|_{\infty} = 2.05 \cdot 10^{-5}, \text{ as well as } \\
    f(\vct{x''}) &= 17485.99 \quad\text{ and }\quad \|G(\vct{x''})\|_{\infty} = 5.85 \cdot 10^{-15}
\end{align*}
where we have abused notation by writing $G(\vct{x}) \coloneqq (g(\vct{x}))_{g \in G}$.
The point $\vct{x''}$ is feasible to nearly machine precision and so we can reasonably conclude $f(\vct{x''}) \geq f_{\set{K}}^\star$.
We combine this with the SAGE bound to obtain $(f_{\set{K}}^\star - f_{\set{K}}^{(2)}) / f_{\set{K}}^\star \leq 0.0013$.
That is, the $\cA$-degree 2 SAGE relaxation solves \eqref{prob:ex2} within one percent relative error.


One can alternatively approach this problem through a global solver from the traditional nonlinear programming community.
We tested BARON, ANTIGONE, LINDO, and SCIP -- which together are four out of the five global nonlinear solvers in the Mittelmann benchmarks.\footnote{The fifth solver (COUENNE) was not available in our version of GAMS (33.2.0).}
We ran each of these solvers by passing it \eqref{prob:ex2} once in variables $\vct{t}$ and once in variables $\vct{x}$.
When passing the problem in variables $\vct{t}$ we had to disable warnings from GAMS about unbounded monomials with negative exponents.
For all configurations we used a time limit of 7200 seconds, allocated 8 threads, and left the machine otherwise unused.

In both parameterizations, SCIP terminated after 7200 seconds with no feasible solution and no lower bound.
Precise results for the remaining solvers are reported  in Table \ref{tab:globalSolvers}.
The overall takeaway is that SAGE produced the same solution as these solvers, but with an REP that could be solved in half the time as the fastest of these methods.
Only LINDO was able to certify its solution as globally optimal.
By contrast with LINDO, the performance of SAGE is independent of whether signomials are considered as generalized polynomials in $\vct{t}$ or as functions of $\vct{x}$.

\begin{table}[ht!]\label{tab:globalSolvers}
    \small\centering\caption{Results of applying global solvers from GAMS to a reactor design problem in chemical engineering \eqref{prob:ex2}. All solvers returned a solution with objective value approximately equal to $17485.99$. }
    \begin{tabular}{l|rr|rr|}
             & \multicolumn{2}{l}{Using $\vct{t}$ as optimization variable} & \multicolumn{2}{|l|}{Using $\vct{x}$ as optimization variable}  \\
             & \multicolumn{1}{l}{solver time (s)} & lower bound                    & \multicolumn{1}{l}{solver time (s)} & lower bound  \\ \hline
    BARON    & 163                                 & $-\infty$                      & 7200                                & $-\infty$      \\
    ANTIGONE & 145                                 & {$-16880.380$} & 7200                                & $-\infty$                     \\
    LINDO    & 1468                                & {$17484.314$}  & 50                                  & {$17485.988$} \\ 
    \end{tabular}
\end{table}

\section{Nonnegativity and signomial moments}\label{sec:sig_moments}
\label{sec:moment}
%
%

In this section we apply basic functional analysis to signomial nonnegativity problems.
We begin with definitions that are analogous to those in the moment-SOS literature for polynomial optimization.
We then state the section's main theorem -- a method to develop successively stronger outer-approximations for the cone of signomials in $\R[\cA]$ that are nonnegative on a compact set $\set{K}$.
Unlike in previous sections, the set $\set{K}$ is not defined by a convex set $\X$ and signomial constraint functions $G$.
In order to prove our approximation result we establish two basic facts regarding existence and uniqueness of representing measures for signomial moment sequences.

\subsection{Definitions}\label{subsec:moment_defs}

Throughout this section and the next, we use $\struc{\cA_{\infty}}$ to denote the smallest subset of $\R^n$ that contains all the $\cA_d$.
It is clear that signomials $f \in \R[\cA]$ are in one-to-one correspondence with finitely supported sequences $\struc{\vct{f}} = (f_{\evec})_{\evec\in\cA_{\infty}}$.
Most of our arguments in this section focus on the dual space to $\R[\cA]$, which we identify with $\R^{\cA_{\infty}}$.
Sequences $\vct{y} \in \R^{\cA_{\infty}}$ are associated to linear functions \struc{$L_{\vct{y}} \colon \R[\cA] \to \R $} defined by $L_{\vct{y}}(\e^{\evec}) = y_{\evec}$.
We call $L_{\vct{y}}$ the \textit{Riesz functional} of $\vct{y}$.

We are most interested in Riesz functionals that are induced by \textit{moment sequences}.
That is, when there is a Borel measure for which
\[
    y_{\evec} = \int \e^{\evec}(\vct{x}) \diff\mu(\vct{x}) \quad \foralll \quad \evec \in \cA_{\infty}
\]
-- in which case we have $L_{\vct{y}}(f) = \int f(\vct{x})\diff\mu(\vct{x})$.
To describe these measures, we use $\mathscr{B}$ to denote the smallest family of subsets of $\R^n$ that contains all compact subsets of $\R^n$ and that is closed under finite union, set-theoretic difference, and countable intersection.
A (finite) Borel measure $\mu$ is a nonnegative set function on $\mathscr{B}$ such that $\mu(\emptyset) = 0$, $\mu(\R^n) < +\infty$, and $\mu(\cup_{i=1}^{\infty} \set{B}_i) = \sum_{i=1}^{\infty} \mu(\set{B}_i)$ for any disjoint collection of sets in $\mathscr{B}$.
The support of such a measure (denoted $\supp\mu$) is the unique smallest closed $\set{B} \in \mathscr{B}$ for which $\mu(\R^n \setminus \set{B}) = 0$.

Next, we introduce a concept directly analogous to the ``localizing matrix'' in the moment-SOS literature.
In our case, the \struc{\textit{localizer}} induced by a sequence $\vct{y} \in \R^{\cA_{\infty}}$ and a dimensional parameter $d \in {\N}\cup\{+\infty\}$ is the linear operator
\[
   \struc{V_d(\ \cdot \ \vct{y})}\colon \R[\cA] \to \R^{\cA_{d}} \qquad \text{ defined by } \qquad f \mapsto  (L_{\vct{y}}(f \e^{\evec}) \,:\, \evec \in \cA_d).
\]
We abbreviate the case $V_d(\e^{\vct{0}}\vct{y})$ by $\struc{V_d(\vct{y})}$.

Localizers help us make abstract arguments concrete.
For example, a localizer can truncate infinite sequences $\vct{y} \in \R^{\cA_{\infty}}$ to $V_d(\vct{y}) = (y_{\evec} \,:\, \evec \in \cA_d)$.
One can also see that if the $\cA$-degree of a signomial $f$ is at most $d$, then we can identify $f$ by a vector of coefficients $\vct{f} \in \R^{\cA_d}$ and evaluate a Riesz functional by $L_{\vct{y}}(f) = \langle \vct{f}, V_{d}(\vct{y}) \rangle$.
That second point is important:
given a convex cone $C \subset \R[\cA]_{d}$, the condition that $V_{d}(\vct{y}) \in C^\dagger$ is equivalent to $L_{\vct{y}}(f) \geq 0$ for all $f \in C$.

\subsection{Approximating nonnegativity cones from the outside}

One of this article's main contributions is to show how essentially any \positiv{} can be turned around to obtain arbitrarily strong \textit{outer approximations} of nonnegativity cones.
The following definition helps us state our results on this topic as well as further results in Section \ref{sec:upper_bounds}.

\begin{definition}\label{def:AK_complete}
    A sequence of closed convex cones $(C_d)_{d \geq 1}$ is called \emph{$(\cA,\set{K})$-complete} if \emph{(i)} $C_d \subset \R[\cA]_{d}$, \emph{(ii)} every $f \in C_d$ is $\set{K}$-nonnegative, and \emph{(iii)} for every $\set{K}$-positive $f \in \R[\cA]$, there exists some $d$ for which $f \in C_d$.
\end{definition}

Here is this section's main theorem.

\begin{theorem}\label{thm:outer_rep}
    Let $\cA$ be injective, $\set{K}$ be compact, and $(C_d)_{d \geq 1}$ be $(\cA,\set{K})$-complete.
    If $\vct{y}$ is the moment sequence of a Borel measure $\mu$ with $\supp \mu = \set{K}$,
    then $f \in \R[\cA]$ is $\set{K}$-nonnegative if and only if $V_{d}(f \vct{y}) \in C_d^\dagger$ for all integers $d \geq 1$.
\end{theorem}

We prove this theorem  in Section~\ref{subsec:proof_of_outer_rep}.
It requires two intermediate results that we establish in Section~\ref{subsec:elementary_signomial_moments}, and otherwise closely follows arguments in \cite{Lasserre2011}.
In order to appreciate Theorem~\ref{thm:outer_rep} it is helpful to consider specific examples of $(\cA,\set{K})$-complete sequences.
Below we provide two examples where the book-keeping in the $(\cA,\set{K})$-complete sequence is relatively simple.
We revisit these examples in Section \ref{sec:upper_bounds}, where we provide precise expressions for dual cones $C_d^\dagger$.

\begin{example}\label{ex:AK_complete_1}
    Let $(\cA, G)$ satisfy the requirements of Theorem~\ref{prop:ord_sage_constr_positiv} (e.g., we must have $\cA \subset \Q^n$ and $|G| \geq 2|\cA|$) and set $\set{K} = \{ \vct{x}\,:\, g(\vct{x}) \geq 0 \text{ for all } g \in G \}$. 
    Consider the parameterized sets of signomials $R_q(G)$ in equation \eqref{eq:def_RqG}.
    For each $d \geq 1$ and $h \in R_d(G)$, we abbreviate $d[h] = d - \deg_{\cA}(h)$ and define
    \begin{align*}
        C_d \coloneqq& \left\{ \textstyle\sum_{h \in R_d(G)} \lambda_h \cdot h \,:\, \lambda_h \in \csage_{\R^n}(\cA_{d[h]}) ~ \forall~ h \in R_d(G) \right\}.
    \end{align*}
    For the purposes of this example one should understand $\cA_{0} = \{\vct{0}\}$ and $\cA_{\ell} = \emptyset$ when $\ell < 0$.
    The sequence $\mathscr{C} = (C_d)_{d\geq 1}$ is $(\cA,\set{K})$-complete. \hfill $\blacklozenge$
\end{example}

\begin{example}\label{ex:AK_complete_2}
    Let $\set{K}$ be a compact convex set.
    To a parameter $d \geq 1$ and a signomial $g \in \R[\cA]_{d}$, associate the cone $C_d(g) = \{ f \in \R[\cA]_{d} \,:\, g f \text{ is } \set{K}\text{-SAGE}\,\}$.
    The sequence of cones defined by Minkowski sums
    \[
        C_d \coloneqq \sum_{\ell=1}^d C_d\left((\textstyle\sum_{\evec\in\cA}\e^{\evec})^{d-\ell}\right)
    \]
    is $(\cA,\set{K})$-complete by Corollary \ref{cor:newPositiv}. \hfill $\blacklozenge$
\end{example}

\begin{corollary}\label{cor:outer_rep:as_cones}
    Let $\cA$ be injective, $\set{K}$ be compact, and $\vct{y}$ be the moment sequence of a Borel measure with support $\set{K}$.
    Suppose $(C_{\ell})_{\ell \geq 1}$ is an $(\cA,\set{K})$-complete sequence that satisfies $C_{\ell} \subset C_{\ell+1}$ for all $\ell$ (such as in Examples \ref{ex:AK_complete_1} or \ref{ex:AK_complete_2}) and denote $C_{\infty} = \cup_{\ell\geq 1} C_{\ell}$.
    If $P_d$ denotes the cone of $\set{K}$-nonnegative signomials in $\R[\cA]_{d}$ and
    \[
        Q_{\ell} \coloneqq \left\{ f \in \R[\cA]_{d} \,:\, V_{\ell}(f\vct{y}) \in C_{\ell}^\dagger \right\}
    \]
    then $Q_{1} \supset Q_{2} \supset \cdots \supset Q_{\infty} = P_d$.
\end{corollary}

\subsection{Elementary results in signomial moment theory}\label{subsec:elementary_signomial_moments}

Here we present a simple result concerning when a sequence $\vct{y} \in \R^{\cA_{\infty}}$ admits a representing measure on some closed set $\set{K}$.
We also present a condition for when the representing measure is unique.

\begin{proposition}\label{prop:sigRHT}
    Suppose $\set{K} \subset \R^n$ is closed and that for every unbounded sequence $(\vct{x}_t)_{t \in \N} \subset \set{K}$, we have $\limsup_{t} \max_{\evec \in \cA} \langle\evec,\vct{x}_t\rangle = +\infty$.
    Given a sequence $\vct{y} \in \R^{\cA_{\infty}}$, there exists a Borel measure $\mu$ on $\set{K}$ such that
    \[
        \int \e^{\evec}(\vct{x}) \diff \mu(\vct{x}) = y_{\evec} \quad \foralll \quad \evec \in \cA_{\infty}
    \]
    if and only if $L_{\vct{y}}(f) \geq 0$ for all signomials $f \in \R[\cA]$ nonnegative on $\set{K}$.
\end{proposition}

Clearly, Proposition~\ref{prop:sigRHT} applies when $\set{K}$ is compact.
What of the noncompact cases?
If $\set{K}$ is convex, then the hypothesis is satisfied if and only if the intersection of the recession cone of $-\cA \set{K}$ and $\R^{\cA}_{+}$ consists only of the origin.
If $(\cA,\set{K})$ satisfy the hypothesis of Proposition~\ref{prop:sigRHT}, then for any $\set{K}'\subset\set{K}$ we know that the proposition also holds for $(\cA,\set{K}')$.
We may conclude that if $\cA$ contains the origin in the interior of its convex hull, then $(\cA,\set{K})$ satisfies the hypothesis of Proposition~\ref{prop:sigRHT} for any $\set{K} \subset \R^n$.

We call a measure \textit{determinate} if it is the unique Borel measure that gives rise to its moment sequence.
It is well known that in the polynomial case, measures supported on compact sets are determinate.
The same is true for signomials.

\begin{proposition}\label{prop:unique_rep_meas}
    Suppose $\cA$ is injective and $\set{K}$ is compact.
    If $\vct{y}$ is a moment sequence of two Borel measures with $\supp \mu_1 \subset \set{K}$ and $\supp \mu_2 \subset \set{K}$, then $\mu_1 = \mu_2$.
\end{proposition}

As we work towards proving the propositions above we cite results from the literature that concern \textit{(locally compact) Hausdorff (topological) spaces}.
General background on these spaces can be found in \cite[\S 21.1, 21.2]{royden2010real} and \cite[\S 37]{simmons1963introduction}.
When we say ``$\set{\Omega}$ is a Hausdorff space'' we mean that $\set{\Omega}$ is a set equipped with some Hausdorff topology.
Our original definition of a Borel measure is adapted to this setting by replacing every appearance of ``$\R^n$'' with ``$\set{\Omega}$'' (including when defining $\mathscr{B}$).
As a matter of notation, we use $\struc{\Cont(\mathsf{\Omega}, \R)}$ denote the ring of all continuous functions $f\colon \mathsf{\Omega} \to \R$.

\begin{theorem}[Theorem 3.1 \cite{marshall2003}]\label{thm:Haviland-help}
    Let $A$ be an $\R$-algebra, $\mathsf{\Omega}$ a Hausdorff space, and
    $\hat \;\colon A \to \Cont(\mathsf{\Omega},\R)$ an $\R$-algebra homomorphism. 
    Assume there exists a $p\in A$ such that $\hat{p}\geq 0$ on $\mathsf{\Omega}$ and, for each integer $\ell\geq 1$, the sublevel set $\mathsf{\Omega}_{\ell}=\{\vct{x}\in \mathsf{\Omega}: \hat{p}(\vct{x})\leq \ell\}$ is compact. 
    Then, for any linear functional $L \colon A \to \R$ satisfying $L(\{a \in A: \hat{a}\geq 0 \text{ on } \mathsf{\Omega}\})\subset \R_+$, there exists a Borel measure $\mu$ on $\mathsf{\Omega}$ such that $L(a) =\int_\mathsf{\Omega} \hat{a} \diff \mu$ for all $a\in A$.
\end{theorem}

\begin{proof}[Proof of Proposition \ref{prop:sigRHT}]
    The claim follows from Theorem \ref{thm:Haviland-help} by simple identifications.
    First, $\set{\Omega} := \set{K}$ is a Hausdorff space when $\set{K}$ is considered in the relative topology induced by the standard topology on $\R^n$. (I.e., where one defines the open subsets of $\set{K}$ to be all sets of the form $\set{K}\cap \set{U}$ for open $\set{U}\subset\R^n$.)
    The remaining identifications are
    $A=\R[\cA]$ and $\hat \;\colon \R[\cA] \to \Cont(\set{K},\R)$ defined by $\hat{f}(\vct{x})=f(\vct{x})$ for all $\vct{x}\in \set{K}$ (i.e., $\hat{f}$ is the restriction of the signomial $f$ to $\set{K}$).
    
    We make use of the distinguished signomial $p = \sum_{\evec \in \cA} \e^{\evec}$.
    It suffices to show that for any $\ell$ the sublevel set $\mathsf{\Omega}_{\ell} = \{ \vct{x} \in \set{K} \,:\, \hat{p}(\vct{x}) \leq \ell \}$ is compact.
    It is obvious that $\mathsf{\Omega}_{\ell}$ is closed.
    Moreover, in order for a point $\vct{x}$ to belong to $\mathsf{\Omega}_{\ell}$ it is necessary that $\langle\evec, \vct{x}\rangle \leq \log \ell$ for all $\evec \in \cA$.
    By the theorem's assumption, any unbounded sequence in $\set{K}$ cannot satisfy this property.
    Therefore all sequences $(\vct{x}_t)_{t \in \N} \subset \mathsf{\Omega}_{\ell}$ are bounded, which implies compactness of $\mathsf{\Omega}_{\ell}$.
\end{proof}

Our proof of Proposition \ref{prop:unique_rep_meas} requires intermediate use of \textit{Radon measures}.
These are the Borel measures which satisfy the property that for every $\set{B} \in \mathscr{B}$ and $\epsilon > 0$, there are sets $\set{A},\set{C} \in \mathscr{B}$ where (i) $\set{A}\subset\set{B}$ is compact, (ii) $\set{C} \supset \set{B}$ is open, and (iii) $\mu(\set{C}) + \epsilon \geq \mu(\set{B}) \geq \mu(\set{A}) - \epsilon$.
In our context, Radon measures are important because of their role in the following theorem (which uses $\Cont_c(\set{\Omega},\R)$ for the ring of compactly supported continuous functions from $\set{\Omega}$ to $\R$).

\begin{theorem}[Riesz-Markov, see \S 21.4 of \cite{royden2010real}]\label{thm:RMT}
    Let $\set{\Omega}$ be a locally compact Hausdorff space and $L$ be a linear functional on $\Cont_c(\set{\Omega},\R)$.
    If $L(f) \geq 0$ for every nonnegative $f \in \Cont_c(\set{\Omega},\R)$,
    then there is a unique Radon measure $\mu$ over $\set{\Omega}$ for which $L(f) = \int_{\set{\Omega}} f \diff\mu$ for every $f \in \Cont_c(\set{\Omega},\R)$.
\end{theorem}

\begin{proof}[Proof of Proposition \ref{prop:unique_rep_meas}]
    We can take $\set{\Omega} = \set{K}$ by giving $\set{K}$ the relative topology induced by the standard topology on $\R^n$.
    Note that $\set{K}$ is actually a compact metric space upon adopting the Euclidean metric, and $\Cont_c(\set{K},\R) = \Cont(\set{K},\R)$.
    Therefore if we can show that $L_{\vct{y}}\colon \R[\cA] \to \R$ has a unique continuous extension $\overline{L}_{\vct{y}}\colon \Cont(\set{K},\R) \to \R$, then the claim will follow for Radon measures by applying Theorem \ref{thm:RMT}.
    
    The uniqueness of such an extension can be stated as follows: for every $\phi \in \Cont(\set{K},\R)$, we have
    \begin{equation}\label{eq:uniqe_rep_meas:need_RRT}
        \sup_{f \in \R[\cA]}\{ L_{\vct{y}}(f) \,:\, \phi - f \geq 0 \text{ on } \set{K} \} = \inf_{f \in \R[\cA]}\{ L_{\vct{y}}(f) \,:\, f - \phi \geq 0 \text{ on } \set{K} \}.
    \end{equation}
    It is easily shown that \eqref{eq:uniqe_rep_meas:need_RRT} holds if every function in $\Cont(\set{K},\R)$ can be approximated to arbitrary precision (in sup norm) by a signomial in $\R[\cA]$.
    The Stone-Weierstrass Theorem tells us that such an approximation exists if signomials in $\R[\cA]$ can separate points, i.e., if for every pair of distinct $\vct{x},\vct{x'}\in \set{K}$, there exists an $f \in \R[\cA]$ for which $f(\vct{x}) \neq f(\vct{x'})$.
    
    We now show that signomials in $\R[\cA]$ can separate points.
    Let $\vct{x}$ and $\vct{x'}$ be distinct points in $\R^n$.
    By the injectivity of $\cA$, the images $\vct{z} \coloneqq \cA \vct{x}$ and $\vct{z'} \coloneqq \cA \vct{x'}$ are likewise distinct in $\R^{\cA}$.
    Recall that these vectors have components $z_{\evec} = \langle \evec,\vct{x}\rangle$ and $z_{\evec}' = \langle \evec,\vct{x'}\rangle$, so the condition that $\vct{z} \neq \vct{z'}$ means there exists a $\bevec \in \cA$ for which $\langle\bevec,\vct{x}\rangle \neq \langle \bevec,\vct{x'}\rangle$.
    We exponentiate both sides of that non-equality to find
    \[
        \e^{\bevec}(\vct{x}) = \exp\langle\bevec,\vct{x}\rangle \neq \exp\langle\bevec,\vct{x'}\rangle = \e^{\bevec}(\vct{x'}),
    \]
    which that tells us that $\e^{\bevec} \in \R[\cA]$ separates $\vct{x},\vct{x'}$.
    
    We complete the proof by noting that since $\set{K}$ is a compact metric space, every Borel measure on $\set{K}$ is also a Radon measure \cite[\S 21.5, Theorem 14]{royden2010real}.
\end{proof}

\subsection{Proof of Theorem \ref{thm:outer_rep}}\label{subsec:proof_of_outer_rep}

    For this proof we denote the cone of $\set{K}$-nonnegative signomials of $\cA$-degree at most $d$ by $P_{d}$.
    Properties (i) and (ii) of the $(\cA,\set{K})$-complete sequence $\mathscr{C}= (C_d)_{d\geq 1}$ tell us that $C_d \subset P_{d}$.
    
    Suppose $f$ is $\set{K}$-nonnegative.
    We will show that $V_{d}(f \vct{y}) \in C_d^\dagger$ holds for all $d$.
    As a first step, define $\vct{\hat{y}} \in \R^{\cA_{\infty}}$ by $\hat{y}_{\evec} = \int \e^{\evec}(\vct{x})f(\vct{x})\diff\mu(\vct{x})$ for all $\evec\in\cA_{\infty}$.
    Because $f$ is $\set{K}$-nonnegative, the differential quantity $\diff\phi(\vct{x}) = f(\vct{x})\diff\mu(\vct{x})$ defines a Borel measure on $\set{K}$, so $\vct{\hat{y}}$ is a moment sequence.
    Meanwhile, the simple identity $L_{\vct{y}}(f \e^{\evec}) = \hat{y}_{\evec}$ tells us that $V_d(f \vct{y}) = V_d(\vct{\hat{y}})$ for all $d$.
    Combine these to see that $V_d(f\vct{y}) \in P_{d}^\dagger$ for all $d$.
    The result follows since $P_{d}^\dagger \subset C_d^\dagger$.
     
    Now we address the theorem's other claim:
    we show that if $V_d(f \vct{y}) \in C_d^\dagger$ for all $d$, then $f$ is nonnegative on $\set{K}$.
    
    Once again, we define $\vct{\hat{y}} \in \R^{\cA_{\infty}}$ by $\hat{y}_{\evec} = \int \e^{\evec}(\vct{x})f(\vct{x})\diff\mu(\vct{x})$ so that $V_d(\vct{\hat{y}}) = V_d(f\vct{y})$.
    Let $d$ be any fixed positive integer.
    We claim that $L_{\vct{\hat{y}}}(g) \geq 0$ for all $g \in P_{d}$;
    by a continuity argument this claim holds if $L_{\vct{\hat{y}}}(g) \geq 0$ for all $\set{K}$-positive $g$ with $\deg_{\cA}(g) \leq d$.
    Let us fix such a $g$.
    By property (iii) of $\mathscr{C}$, there exists an integer $d' \geq \deg_{\cA}(g)$ for which $g \in C_{d'}$.
    Now, our assumption on $\vct{\hat{y}}$ includes $V_{d'}(\vct{\hat{y}}) \in C_{d'}^\dagger$, which tells us $L_{\vct{\hat{y}}}(g) \geq 0$.
    Therefore $L_{\vct{\hat{y}}}(g) \geq 0$ for every $g \in P_{d}$ for our \textit{arbitrary} fixed $d$.
    We can now invoke Proposition \ref{prop:sigRHT} to see that there is some Borel measure $\psi$ with $\supp \psi \subset \set{K}$ and $\vct{\hat{y}}$ as its moment sequence.
    Since $\set{K}$ is compact and $\cA$ is injective, Proposition \ref{prop:unique_rep_meas} tells us that $\psi$ is unique.
    
    By now we have shown that there is a unique Borel measure $\psi$ for which
    \begin{equation}\label{eq:outer_rep:moment_match}
            \int \e^{\evec}(\vct{x})f(\vct{x})\diff\mu(\vct{x}) = \int\e^{\evec}(\vct{x})\diff\psi(\vct{x}) \quad \text{ for all }\quad \evec \quad\text{ in }\quad\cA_{\infty}.
    \end{equation}
    Getting from \eqref{eq:outer_rep:moment_match} to ``$f \geq 0$ on $\set{K}$'' requires two steps.
    The main step is to carefully use moment determinacy (Proposition \ref{prop:unique_rep_meas}) to show that $f \diff \mu$ induces a Borel measure on $\set{K}$.
    The second step is to invoke \cite[Lemma 3.1]{Lasserre2011}, which tells us that $f$ (as a continuous function on a separable metric space) is nonnegative on $\set{K} = \supp \mu$ if and only if the set function $\set{B} \mapsto \int_{\set{K}\cap\set{B}} f \diff\mu$ is a positive measure (e.g., a Borel measure).
    
    So we turn to showing that $f \diff \mu$ induces a Borel measure.
    Begin by introducing $\set{B}_1 = \{ \vct{x} \in \set{K}\,:\, f(\vct{x}) \geq 0\}$ and $\set{B}_2 = \{ \vct{x} \in \set{K} \,:\, f(\vct{x}) < 0 \}$.
    We want to show that $\set{B}_2$ is empty but we have no tools to do this directly.
    Instead, we use $\set{B}_1, \set{B}_2$ to define the functions
    \[
        \phi_1(\set{B}) = \int_{\set{B} \cap \set{B}_1} f(\vct{x})\diff\mu(\vct{x}) \quad\text{ and }\quad \phi_2(\set{B}) = \int_{\set{B} \cap \set{B}_2} (-f(\vct{x}))\diff\mu(\vct{x}). 
    \]
    These functions are Borel measures since $f$ is continuous and $\set{K}$ is compact.
    We can therefore define what is known as a \textit{signed measure} (see \cite[\S 17.2]{royden2010real}) $\phi = \phi_1 - \phi_2$ and note that $\int\e^{\evec}(\vct{x})\diff\phi(\vct{x}) = \int\e^{\evec}(\vct{x})\diff\psi(\vct{x})$ for all $\evec\in \cA_{\infty}$ -- equations that can be rewritten as 
    \[
        \int \e^{\evec}(\vct{x})\diff\phi_1(\vct{x}) = \int\e^{\evec}(\vct{x})\diff(\psi + \phi_2)(\vct{x}) \quad \text{ for all } \evec \quad\text{ in }\quad\cA_{\infty}.
    \]
    The key is that now, $\phi_1$ and $\psi + \phi_2$ are Borel measures, therefore the fact that their moments match lets us use Proposition \ref{prop:unique_rep_meas} to conclude that they are unique, i.e., $\phi_1 = \psi + \phi_2$.
    From here we simply rewrite $\phi = \phi_1 - \phi_2 = \psi$ to see that since $\psi$ is a Borel measure, so is $\phi$.
    The result follows from \cite[Lemma 3.1]{Lasserre2011} since $\diff\phi = f\diff\mu$.

\section{Complete hierarchies of upper bounds}\label{sec:upper_bounds}

The following result states that essentially any SAGE-based \positiv{} can be converted into a complete hierarchy of upper bounds for signomial minimization.

\begin{theorem}\label{thm:generalUB}
    Let $\cA$ be injective, $\set{K}$ be compact, and $\mu$ be a Borel measure with support $\set{K}$.
    Consider an $(\cA,\set{K})$-complete sequence $\mathscr{C} = (C_d)_{d \geq 1}$ where $C_d \subset C_{d+1}$ and each $C_d$ contains all posynomials in $\R[\cA]_{d}$.
    For $f \in \R[\cA]$ and integers $d \geq 1$, define
    \begin{equation}\label{eq:abstract_upper_def}
        \theta_{d} =  \inf_\psi\left\{ \int f \psi\diff\mu \,:\, \int \psi\diff \mu= 1,\, \psi \in C_d \right\}.
    \end{equation}
    The sequence $(\theta_{d})_{d \geq 1}$ monotonically converges to $f_{\set{K}}^\star \coloneqq \min\{ f(\vct{x})\,:\, \vct{x}\in\set{K}\}$ from above.
\end{theorem}

One calls $\mu$ in Theorem \ref{thm:generalUB} a \textit{reference measure}.
The validity of the hierarchy follows from the facts that (1) $\psi$ and $\mu$ induce a probability measure $\diff\phi = \psi\diff\mu$ on $\set{K}$, and (2) the objective function in \eqref{eq:abstract_upper_def} is simply the average of $f$ according to $\phi$.
The hierarchy's convergence property qualitatively states that given any reference measure $\mu$ with compact support $\set{K}$, Dirac distributions centered on a signomial's minimizer(s) over $\set{K}$ can be approximated to arbitrary accuracy by distributions of the form ``$\psi\diff\mu$'' for signomials $\psi$ that are nonnegative on $\set{K}$.

The original idea for this approach comes from Lasserre's \cite{Lasserre2011}.
We were made aware of this idea in a broader sense through a presentation by de Klerk at the 2019 ICCOPT meeting in Berlin, Germany (see \cite{dKLLS2017,dKLS2017-convergence}).
Lasserre's approach maps to Theorem \ref{thm:generalUB} in the sense that $\R[\cA]_d$ is replaced by the polynomial ring in $n$ variables and $C_d$ is the cone of polynomials in $n$ variables of degree $2d$ that admit SOS decompositions.
Note in particular that the cones $C_d$ in \cite{Lasserre2011} consist of only a very specific class of globally nonnegative polynomials.
The approach in \cite{dKLLS2017} is closer to our method.
Besides working in the polynomial ring, it takes for $C_d$ the cone of polynomials of degree $d$ that admit a certificate of nonnegativity over $[0,1]^n$ via the Handelman hierarchy; see \cite{cassier1984probleme,handelman1988representing}.

This article does not report on numerical experiments with a hierarchy of the form in Theorem \ref{thm:generalUB}.
Rather, we explain the basic considerations needed to implement such a hierarchy (Section~\ref{subsec:principles_upper_bounds}) and detail the construction of two such hierarchies (Section~\ref{subsec:example_upper_hierarchies}), before ultimately proving the theorem (Section~\ref{subsec:proof_lower_hier}).

\subsection{General principles in computing the upper bounds}\label{subsec:principles_upper_bounds}

Our goal here is to express problem \eqref{eq:abstract_upper_def} using suitable numeric problem data.
We begin by using $\vct{y} \in \R^{\cA_{\infty}}$ to denote the moment sequence of $\mu$.
Next, let $\vct{\psi} = (\psi_{\evec}\,:\,\evec\in\cA_{d})$ be the vector of coefficients for a signomial $\psi = \sum_{\evec\in\cA_{d}}\psi_{\evec}\e^{\evec} \in C_d$.
The objective function of \eqref{eq:abstract_upper_def} is
\[
     \int f \psi\diff\mu = \sum_{\evec\in\cA_{d}}\psi_{\evec }\int f \e^{\evec}\diff\mu = \sum_{\evec\in\cA_d} \psi_{\evec}L_{\vct{y}}(f\e^{\evec}) = \langle\vct{\psi},V_{d}(f \vct{y})\rangle.
\]
We expanded $\psi$ as a sum and used linearity of integration to obtain the first equality above.
The second equality used the fact that Riesz functionals of moment sequences simply perform integration (i.e., $L_{\vct{y}}(h) = \int h\diff\mu$).
Finally, the third equality applied the definition of signomial moment localizers.
One may similarly verify that the constraint $\int\psi\diff\mu = 1$ can be written as $\langle\vct{\psi},V_{d}(\vct{y}) \rangle = 1$.

We take this space to record the following representation for \eqref{eq:abstract_upper_def}
\begin{equation}\label{eq:upperBoundHierDef}
    \inf\left\{ \,\langle V_{d}(f\vct{y}), \vct{\psi} \rangle \,:\, \langle V_{d}(\vct{y}), \vct{\psi} \rangle = 1,~ \psi \in C_d \right\}
\end{equation}
and a representation of its conic dual
\begin{equation}\label{eq:sage:rep1}
    \sup\left\{ \theta \,:\, V_{d}(f\vct{y}) - \theta V_{d}(\vct{y}) \in C_d^\dagger,~ \theta \in \R \right\}.
\end{equation}
Strong duality holds for this pair of problems under the (essentially universal) condition that every posynomial in $\R[\cA]_{d}$ is also in $C_d$; see Proposition \ref{prop:strongdual}.

In order to have a hope of solving these problems, we must be able to compute the localizers $V_{d}(f\vct{y}), V_{d}(\vct{y}) \in \R^{\cA_d}$.
Doing this is equivalent to evaluating appropriate signomial moments $y_{\evec} =\int \e^{\evec}\diff\mu$.
In particular, $V_d(f\vct{y})$ is a linear function of moments with $\cA$-degree up to $d + \deg_{\cA}(f)$.

There are only a few interesting cases where these moments can be derived either in closed form or numerically.
An especially prominent case is when $\mu$ is the uniform measure on a box.
Other examples with closed-form expressions for signomial moments include uniform measures over ellipsoids  \cite[Theorem 3.2]{LZ2001} and solid simplices \cite[Theorem 2.6]{LZ2001}.
If $\mu$ is the uniform measure on a polytope then one can nominally compute moments by triangulating that polytope with simplices (see \cite{Barvinok:Volume}).

Once we have the necessary localizers, computing the upper bounds $\theta_{d}$ requires a method to optimize a linear function over an affine slice of $C_d$ or to check membership in $C_d^\dagger$.
The former approach has the advantage of recovering the distribution $\psi\diff\mu$ that may contain information on the minimizer of $f$ over $\set{K}$.
The latter approach has the benefit of being amenable to simpler optimization algorithms.
The possibility of using simpler algorithms is preferable for reasons other than efficiency.
Except in special cases on $(\cA,\set{K})$, the $\ell_{\infty}$ norms of the localizers $V_d(f\vct{y}),V_d(\vct{y})$ grow exponentially with $d$, and this can create numerical trouble even for reliable optimization solvers.

\subsection{Specific hierarchies for box-constrained problems}\label{subsec:example_upper_hierarchies}

The main obstacle in implementing the hierarchy from Theorem \ref{thm:generalUB} is computing the moment sequence of the reference measure $\mu$.
We consider here the special case where $\mu$ is the uniform measure over a box $\set{K} = [\ell_1,u_1] \times \cdots \times [\ell_n, u_n]$ and approach the problem through SAGE-based methods.

\begin{example}\label{ex:AK_complete_1_next}
    We continue from Example \ref{ex:AK_complete_1}.
    In order for that example to apply to our box-constrained problem, we assume the standard basis vectors $\vct{\delta}_{i} \in \R^n$ belong to $\emat$ and set
    \[
    G = \{ 1\} \cup \{ \e^{\evec}(\vct{u}) - \e^{\evec},~ \e^{\evec} - \e^{\evec}(\vct{\ell}) \,:\, \evec\in\cA\}.
    \]
    The sequence $\mathscr{C} = (C_d)_{d\geq 1}$ is $(\cA,\set{K})$-complete by Theorem~\ref{prop:ord_sage_constr_positiv}.
    It is also clearly nested and each constituent cone contains all posynomials up to a given degree (since $1 \in G$).
    Therefore the upper bounds from Theorem \ref{thm:generalUB} converge to $f_{\set{K}}^\star$ for any $f \in \R[\cA]$.
    
    Now we speak to how the bounds might be computed by solving \eqref{eq:sage:rep1} with bisection on $\theta$.
    Recalling the families of signomials $R_d(G)$ from equation \eqref{eq:def_RqG} and again abbreviating $d[h] = d - \deg_{\cA}(h)$, the dual cone to each $C_d$ is
    \begin{align*}
        C_d^\dagger = \left\{ \vct{v} \in \R^{\cA_d}_+ \,:\, V_{d[h]}(h\vct{v}) \in \csage_{\R^n}\left(\cA_{d[h]}\right)^\dagger ~ \forall~ h \in R_d(G)\right\}.
    \end{align*}
    When interpreting the above expression one should drop all constraints where $d[h] < 0$ and define $V_0(f\vct{v}) = L_{\vct{v}}(f)$.
    Notice how the problem of checking membership in $C_d^\dagger$ decouples over each $h \in R_d(G)$.
    Furthermore, checking membership in $\csage_{\R^n}(\cA_{d[h]})^\dagger$ decouples further to consider one dual AGE cone at a time.
    Finally, membership $\vct{v} \in \csage_{\R^n}(A,\bevec)^\dagger$ can be decided by solving a linear programming feasibility problem where a variable $\vct{z} \in \R^n$ satisfies $|A|$ inequalities (see equation \eqref{eq:represent_dual_x_age}). \hfill $\blacklozenge$
\end{example}

\begin{example}\label{ex:AK_complete_2_next}
    We revive Example \ref{ex:AK_complete_2} while assuming that $\cA$ is injective.
    It is easily verified that the provided sequence $(C_d)_{d \geq 1}$ satisfies the hypothesis of Theorem~\ref{thm:generalUB},
    therefore the associated upper bounds converge to $f_{\set{K}}^\star$.
    We turn to describing
    the necessary dual cones $C_d^\dagger$ for solving \eqref{eq:sage:rep1} by bisection.
    For $1 \leq j \leq d$, define $w_j = (\sum_{\evec \in \cA} \e^{\evec})^{d-j}$.
    We have
    \[
    C_d^\dagger = \left\{ \vct{v} \in \R^{\cA_{d}} \,:\, \forall j \leq d, \exists \vct{y}_j \in \Csage_{\set{K}}(\cA_{d})^\dagger \text{ where } V_j(\vct{v}) = V_j\left(w_j \vct{y}_j\right) \right\}.
    \]
    From this expression we see that the membership problem decomposes over $j \in [d]$.
    Each constituent feasibility problem ``$\vct{y}_j \in C_d^\dagger, V_j(\vct{v}) = V_j\left(w_j \vct{y}_j\right)$'' can be approached with any algorithm suitable for optimizing over $\csage_{\set{K}}(\cA_d)^\dagger$. \hfill $\blacklozenge$
\end{example}

\subsection{Proof of Theorem \ref{thm:generalUB}}\label{subsec:proof_lower_hier}

We need the following proposition.

\begin{proposition}\label{prop:strongdual}
    Let $\set{K}\subset\R^n$ be closed.
    Consider a closed convex cone $C_d \subset \R[\cA]_{d}$ of $\set{K}$-nonnegative signomials that contains all posynomials in $\R[\cA]_{d}$.
    For any Borel measure $\mu$ with support contained in $\set{K}$ and associated moment sequence $\vct{y} \in \R^{\cA_{\infty}}$, the pair \eqref{eq:sage:rep1}-\eqref{eq:upperBoundHierDef} exhibits strong duality.
    If $f_{\set{K}}^\star > -\infty$ and $\supp \mu$ has nonempty interior, then \eqref{eq:upperBoundHierDef} attains an optimal solution.
\end{proposition}
\begin{proof}
    First we prove strong duality in the sense of objective values.
    Since $\vct{y}$ is a moment sequence, we have that $V_{d}(\vct{y}) = (y_{\evec} \,:\,\evec \in \cA_{d})$ is elementwise positive.
    It is obvious that the posynomial $\psi = \sum_{\evec \in \cA_{d}} \e^{\evec}$ belongs to the interior of $C_d$, and so defining $s \coloneqq \langle V_{d}(\vct{y}), \vct{\psi} \rangle > 0$, the signomial $\psi' = \psi / s$ is strictly feasible for \eqref{eq:upperBoundHierDef}.
    The claim follows by invoking Slater's condition.
    
    Now suppose $f_{\set{K}}^\star > -\infty$ and $\set{K}$ has nonempty interior.
    Then for any $\theta < f_{\set{K}}^\star - 1$ and every nonzero function $\psi \in C_d$,
    we have
    \[
        \langle V_{d}((f - \theta)\vct{y}), \vct{\psi} \rangle = \int \underbrace{(f(\vct{x}) - \theta)}_{> 1}\psi(\vct{x})\diff\mu(\vct{x}) > \int \psi(\vct{x})\diff\mu(\vct{x}) > 0.
    \]
    By the above inequalities, we have that $V_{d}((f - \theta)\vct{y}) = V_d(f\vct{y}) - \theta V_{d}(\vct{y})$ belongs to the interior of $C_d^\dagger$.
    Therefore \eqref{eq:sage:rep1} is strictly feasible, and by Slater's condition \eqref{eq:upperBoundHierDef} attains an optimal solution.
\end{proof}

Now, let $\cA$, $\set{K}$, $\mu$, $\mathscr{C} = (C_d)_{d\geq 1}$, and $\theta_{d}$ be as in the statement of Theorem \ref{thm:generalUB}.
It is easy to see that $\theta_{d}$ is decreasing in $d$:
since $C_d \subset C_{d+1}$ for all $d$, the size of the feasible set in the defining minimization problem \eqref{eq:abstract_upper_def} is increasing in $d$.

Since $\set{K}$ is compact and $f$ is continuous, we know that $f_{\set{K}}^\star$ is a real number.
The sequence $(\theta_{d})_{d \geq 1}$ is therefore decreasing and bounded below by $f_{\set{K}}^\star$.
Let $\theta^\star$ denote the limit of this sequence and suppose that $\theta^\star > f_{\set{K}}^\star$.
By consideration to the formulation \eqref{eq:sage:rep1}, this limit satisfies
\[
    V_d((f-\theta^\star)\vct{y}) \in C_d^\dagger \quad \text{ for all } \quad d \geq 1.
\]
But then by Theorem \ref{thm:outer_rep}, we have that $f - \theta^\star$ is nonnegative on $\set{K}$.
This contradicts our earlier assumption that $\theta^\star > f_{\set{K}}^\star$, and so we must have $\theta^\star = f_{\set{K}}^\star$, and this completes our proof of Theorem \ref{thm:generalUB}.

\begin{remark}
    Note that in general the convergence is only asymptotic and not finite. That is, in general, there does not exist a finite $d$ such that $\theta_{d} = f_{\set{K}}^\star$.
\end{remark}

\section{Conclusion}\label{sec:outlook}

In this article we have provided the most general \positiv{} to-date for signomials over compact sets.
We designed a complete hierarchy of lower bounds for signomial minimization based on relative entropy programming, which is the first-of-its-kind in that it respects the structure of signomial rings.
We provided a language for and basic results in signomial moment theory, and we used that theory to turn (hierarchical) inner-approximations of nonnegativity cones into (hierarchical) outer-approximations of the same.
Finally, we explained how any such construction leads to highly structured convex programs which produce bounds that approach a signomial's minimum from above.

A great many questions remain for understanding the power and limitations of these methods.
We take this space to record some specific suggestions for lines of future work.

\textit{Strength of the lower bounds.} When working in naive rings, the lowest level relaxation of our hierarchy will always be at least as strong as the lowest levels of the hierarchies in \cite{CS2016,MCW2019}.
For a given ring $\R[\cA]$, what can be said about the problem data $(f,G,\X)$ for which the lowest level of the hierarchy computes $f_{\set{K}}^\star$ exactly?
Alternatively, for given problem data $(f,G,\X)$, how can one choose the ring $\R[\cA]$ so that the proposed hierarchy exhibits ``good'' performance?

\textit{Moment problems.}
Solution recovery with dual formulations to our hierarchy of lower bounds uses candidate points $\vct{x} = \vct{z} / v_{\bevec}$ from dual AGE cones (see  \eqref{eq:represent_dual_x_age}). What is the best way to synthesize this information across the many dual AGE cones that compose a given problem?
This question was previously posed in \cite{MCW2019}, but in the context of a far more complicated (and heuristically designed) hierarchy of lower bounds.
More abstractly, we ask, how can one efficiently recover representing measures from truncated (finite) signomial moment sequences?

\textit{Detailed study of the upper bounds.} We have only provided the generic construction for hierarchies of upper bounds. For a specific construction (e.g., Example \ref{ex:AK_complete_1_next}) what is the convergence rate as a function of $d$? Is this rate dependent on whether or not $\cA$ contains the origin in the interior of its convex hull, as Proposition \ref{prop:sigRHT} might suggest?
As a separate line of questioning:
can one devise a complete SAGE-based hierarchy of upper bounds that does not rely on $(\cA,\set{K})$-complete sequences?
When $\set{K}$ is a compact convex set, one might try using \eqref{eq:abstract_upper_def} where $C_d$ is simply the cone of $\set{K}$-SAGE signomials in $\R[\cA]_d$.

\bibliography{references.bib}
\bibliographystyle{plain}

\end{document}